\documentclass[10pt, a4]{amsart}

\usepackage[hmargin=1.5in,vmargin=1in]{geometry}
\allowbreak

\numberwithin{equation}{section}

\usepackage{amsfonts,amssymb,amscd,amsmath,latexsym,amsbsy}
\usepackage{graphicx}
\usepackage{amssymb}
\usepackage{amsfonts}
\usepackage{latexsym}
\usepackage{mathrsfs}
\usepackage{mathtools}
\usepackage{amsthm}
\usepackage{comment}
\usepackage{listings}
\usepackage{tikz-cd}
\usepackage{hyperref}
\usepackage{amsmath}

\allowdisplaybreaks

\theoremstyle{plain}
\newtheorem{thm}{Theorem}[section]
\newtheorem{lmm}[thm]{Lemma}
\newtheorem{prop}[thm]{Proposition} 
\newtheorem{coro}[thm]{Corollary}

\theoremstyle{definition}
\newtheorem{defn}[thm]{Definition}
\newtheorem*{rmk*}{Remark} 
\newtheorem{rmk}[thm]{Remark}

\usepackage{mathtools}

\numberwithin{equation}{section}

\newcommand{\R}{\ensuremath{\mathbb{R}}}
\newcommand{\Z}{\ensuremath{\mathbb{Z}}}

\newcommand{\supp}{\ensuremath{\mathrm{supp}}}

\newcommand{\mk}{\ensuremath{\mathfrak}}
\newcommand{\tb}{\ensuremath{\textbf}}

\newcommand{\la}{\ensuremath{\langle}}
\newcommand{\ra}{\ensuremath{\rangle}}

\newcommand{\Id}{\ensuremath{\mathrm{Id}}}
\newcommand{\oc}{\ensuremath{\mathrm{1c}}}

\newcommand{\modulesp}{\ensuremath{\mathcal{W}}}

\newcommand\ubar{\overline{u}}

\newcommand{\SE}{\mathcal E}
\newcommand{\SF}{\mathcal F}

\newcommand{\SM}{\mathcal M}
\newcommand{\SN}{\mathcal N}
\newcommand{\SP}{\mathcal P}

\newcommand{\SX}{\mathcal X}
\newcommand{\SY}{\mathcal Y}

\newcommand\ang[1]{\langle #1 \rangle}

\newcommand\Poip{\mathcal{P}_+}

\newcommand\Fic{\mathcal{F}_{\mathrm{1c}}}
\newcommand\xic{\xi_{\mathrm{1c}}}
\newcommand\etac{\eta_{\mathrm{1c}}}
\newcommand\Psisc{\Psi_{\mathrm{sc}}}

\renewcommand\Re{\operatorname{Re}}
\renewcommand\Im{\operatorname{Im}}

\title[Final state problem for NLS]{The final state problem for the nonlinear Schr\"odinger equation in dimensions 1, 2 and 3}

\author{Andrew Hassell, Qiuye Jia}
\date{\today}

\begin{document}
\maketitle

\begin{abstract}
In this article we consider the defocusing nonlinear Schr\"odinger equation, with time-dependent potential, in space dimensions $n=1, 2$ and $3$, with nonlinearity $|u|^{p-1} u$, $p$ an odd integer, satisfying $p \geq 5$ in dimension $1$, $p \geq 3$ in dimension $2$ and $p=3$ in dimension $3$. We also allow a metric perturbation, assumed to be compactly supported in spacetime, and nontrapping. 
We work with module regularity spaces, which are defined by regularity of order $k \geq 2$ under the action of certain vector fields generating symmetries of the free Schr\"odinger equation. We solve the large data final state problem, with final state in a module regularity space, and show convergence of the solution to the final state.  
\end{abstract}

\tableofcontents

\section{Introduction}\label{sec:intro}

\subsection{Nonlinear Schr\"odinger equation and the final state problem}
In this article we consider the final state problem for the nonlinear Schr\"odinger equation (NLS) in 1, 2, or 3 space dimensions. More specifically we consider 
the PDE 
\begin{align} \label{eq:NLS}
Pu := (D_t+\Delta + V)u &= \pm |u|^{p-1} u , \\ D_t &= -i \partial_t, \quad \Delta = \sum_{i=1}^n D_{z_i} D_{z_i} = - \sum_{i=1}^n \frac{\partial^2 }{\partial z_i^2}, 
 \end{align}
where $u = u(t, z)$ is a function on $\R \times \R^n$, $n = 1, 2$ or $3$, $V = V(t, z)$ is a real-valued time-dependent potential function (from a suitable function space), and $p$ is an odd integer as follows: 
\begin{equation}\begin{aligned}
 &\bullet \text{ if $n=1$, then $p \geq 5$;} \\
&\bullet \text{ if $n=2$ then $p \geq 3$;} \\
&\bullet \text{ if $n=3$ then $p=3$.}
\end{aligned}\label{eq:np}\end{equation}

This range of $n$ and $p$ ensures that $\frac{n}{2} < \frac{p+1}{p-1}$, which will ensure that a crucial Strichartz inequality is valid in Proposition~\ref{prop:Linftybound} --- see Remark~\ref{rem:upperboundp}.  We also need $p \geq 1 + 4/n$ in the proof of Lemma~\ref{lem:NLSLrdecay}. These two conditions determine the set of $(n, p)$ as in \eqref{eq:np} that we are able to treat. In the case where $V \equiv 0$ and there is an exact scaling for the equation, our condition is that the scaling-invariant Sobolev exponent $s_c$ satisfies $0 \leq s_c < 1$; that is, we consider nonlinearities that are energy-subcritical and mass-supercritical or critical (see \cite[Section 3.1]{tao2006nonlinear}). 
Notice that the nonlinearity is required to be a monomial in $u$ and $\ubar$; this is crucial for our methods, which require multiplication (algebra) results for certain function spaces described below. We mostly consider the defocusing case, which is the $-$ sign in \eqref{eq:NLS}; however, some intermediate results apply also to the focusing case. 

In fact, we will also consider the case of variable metrics on $\R^n$. Let $g(t)$ be a family of metrics on $\R^n$, depending smoothly on $t$, so that $g(t) - \delta$ is compactly supported in spacetime. In other words, $g(t) = \delta$ when $|t|$ is large, and also $g(t) = \delta$ when $|z|$ is sufficiently large. In Section~\ref{sec:metric} we consider the PDE 
\begin{equation}
 \label{eq:NLSvarg}
Pu := (D_t+\Delta_{g(t)} + V)u = \pm |u|^{p-1} u , 
 \end{equation}
where $\Delta_{g(t)}$ is the Laplace-Beltrami operator corresponding to the metric $g(t)$. However, our main result is still interesting without a metric perturbation, and in this introduction we shall discuss exclusively the case when $g(t) = \delta$. We prove the main theorem first in this case, and indicate the extra arguments required for the general case in Section~\ref{sec:metric}. 

The \emph{final state problem} for NLS is to solve the PDE \eqref{eq:NLS} subject to an asymptotic condition as $t \to +\infty$ or $t \to -\infty$. Recalling that, for smooth enough data, such as an initial condition in the Schwartz class, the solution of the linear equation 
$$
(D_t + \Delta) u = 0
$$
has a leading asymptotic as $t \to \pm \infty$ of the form 
\begin{equation}\label{eq:finalstatef}
u \sim (4\pi it)^{-n/2} e^{i|z|^2/4t} f\big(\frac{z}{2t} \big),
\end{equation}
where $f$ is a function of the variable $\zeta := z/2t$. Indeed, in this constant coefficient case, $u$ can be expressed explicitly as the `Poisson operator' $\SP_0$ applied to $f$:
\begin{align} \label{eq: Poisson definition}
\mathcal{P}_0f(t,z) = (2\pi)^{-n} \int e^{-it|\zeta|^2}e^{iz \cdot \zeta}f(\zeta)d\zeta.
\end{align}
$\mathcal{P}_0f$ is the unique solution to $(D_t+\Delta)u=0$ with final state data as $t \to \pm \infty$ both equal to $f$, in the sense that its leading term in the asymptotic expansion as $t \rightarrow \pm \infty$ is \eqref{eq:finalstatef}. 
This can be seen by applying the stationary phase lemma to the integral \eqref{eq: Poisson definition}, for smooth enough $f$, where the critical point of the phase is $\zeta = \frac{z}{2t}$. For a more precise statement about this convergence, see Theorem~\ref{thm:main1}.

This convergence can be viewed geometrically on the radial compactification of $\R^{n+1}$, where the $t$-axis is vertical for convenience. Then $\zeta = z/2t$ is a coordinate on the boundary of the compactification either on the open northern hemisphere, or the open southern hemisphere. See Figure~\ref{fig:compact}. The statement thus becomes that, after multiplication by $t^{n/2} e^{-i|z|^2/4t}$, the function $u$ has a continuous extension to the boundary of the compactification.

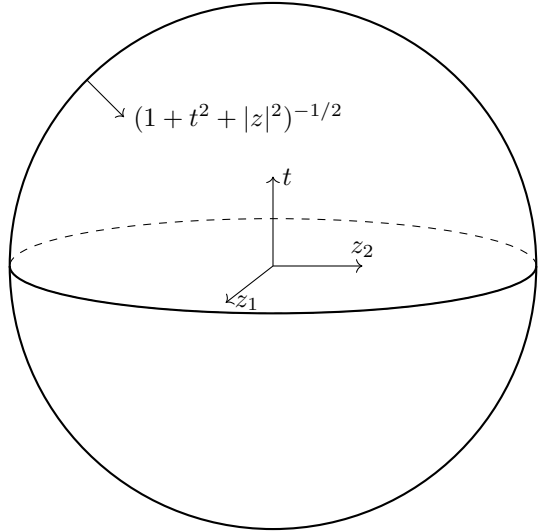
\begin{figure}\label{fig:compact}
	\centering
	\begin{tikzpicture}[scale=0.7]
		\draw[thick] (5,0) arc (0:360:5cm);
		\draw[dashed] (0,0) ellipse (5cm and 0.9cm);
		\draw[thick] (-5,0) arc (180:360:5cm and 0.9cm);
		\draw[->] (-3.536,3.536)--(-3.536*0.8,3.536*0.8) node[anchor=west]  {$(1 + t^2 + |z|^2)^{-1/2}$};
		
		\draw[->] (6-6,-1.5+1.5)--(6-6,0.2+1.5) node[anchor=west] {$t$};
		\draw[->] (6-6,-1.5+1.5)--(5.1-6,-2.2+1.5) node[anchor=west] {$z_1$};
		\draw[->] (6-6,-1.5+1.5)--(7.7-6,-1.5+1.5) node[anchor=south] {$z_2$};
	\end{tikzpicture}
	\caption{The radially compactified spacetime, with time axis oriented vertically. The boundary, which is an $S^n$, is referred to as spacetime infinity, and consists of an open `northern hemisphere', where $t = +\infty$, an open `southern hemisphere', where $t = -\infty$, and the equator, which is the indicated great circle. Any time slice $t = \mathrm{constant}$ meets spacetime infinity at the equator.  The smooth function
		$(1 + t^2 + |z|^2)^{-1/2}$ defines (i.e.\  vanishes simply at) spacetime
		infinity. The function $\zeta := z/(2t)$ extends to a smooth coordinate either on the open northern hemisphere or the open southern hemisphere. It is infinite at the equator.}
	\label{fig1}
\end{figure}

 The final state problem for the PDE \eqref{eq:NLS} is then to find $u$ satisfying this PDE with the asymptotic condition 
\begin{equation}
\lim_{t \to +\infty} (4\pi it)^{n/2} e^{-it|\zeta|^2} u(2t\zeta, t) = f_+(\zeta), \quad \text{(notice that $e^{-it|\zeta|^2} = e^{-i|z|^2/4t}$),}
\label{eq:fspos}
\end{equation}
for some prescribed function $f_+$. We could alternatively prescribe data at $t = -\infty$, in which case the condition becomes 
\begin{equation}
\lim_{t \to -\infty} (4\pi it)^{n/2} e^{-it|\zeta|^2} u(2t\zeta, t) = f_-(\zeta). 
\label{eq:fsneg}
\end{equation}
In this paper we choose to prescribe data at $t = +\infty$. Our main goal in this article is to treat the case of large data: that is, $f$ will lie in a certain function space but its norm in this function space will be unrestricted. In this respect, our results go beyond the results of \cite{gell2023scattering}, which deal only with small final state data. Large data will require us to restrict to the defocusing case.

\subsection{Main result} In this introduction we give a partial introduction to the function spaces used in our main result. A more detailed exposition is given in Section~\ref{sec: module regularity spaces}. 

The main idea of the present paper is to use a time-dependent family of function spaces to measure $u$. We present this in the language of modules, as developed in \cite{HMV2004, gell2022propagation, gell2023scattering} for example. It can also be viewed as a development of Klainerman's vector field method \cite{klainerman1985} as applied to the Schr\"odinger equation. 

We consider the following set of vector fields: 
\begin{align} \label{eq: generator Mt}
 \mathcal{M}_t: \{\Id, z_jD_{z_i}-z_iD_{z_j}, 2tD_{z_j}-z_j, D_{z_j}\}, \quad D_{z_j} = -i \partial_{z_j}. 
\end{align}

This set of vector fields (plus the identity) is closed under commutators, and generates a finite-dimensional group of symmetries of the free Schr\"odinger equation. We use these vector fields to define a module regularity space of order $k$, denoted $\modulesp^k_{\SM_t}$, depending on the parameter $t$. This consists of those functions $u \in L^2$ such that $A_1 \dots A_k u \in L^2$ for every choice of $A_i \in \SM_t$, $1 \leq i \leq k$, with norm 
\begin{align}
||u||_{\modulesp^k_{\SM_t}}^2:= \sum_{A_1, \dots, A_k \in {\SM_t}} ||A_1 \dots A_k u||^2_{L^2(\R^n)}.
\label{eq:WMt} \end{align}
These norms are all equivalent for different values of $t$, but are not uniformly so as $t \to \infty$.

We also consider the module obtained by conjugating each element of $\SM_t$ by $e^{i|z|^2/4t}$, for $t \neq 0$. This is the module 
\begin{align} \label{eq: generator Mt0}
\mathcal{M}_{t,0}: \{\Id, z_jD_{z_i}-z_iD_{z_j}, 2tD_{z_j}, D_{z_j}+\frac{z_j}{2t}\}, \quad t \neq 0.
\end{align}
(The issue with small $t$ is resolved in Section~\ref{sec: module regularity spaces}; we completely ignore this issue in the Introduction.) 
Similarly to \eqref{eq:WMt}, we define $\modulesp^k_{\SM_{t,0}}$ by 
\begin{align}
||u||_{\modulesp^k_{\SM_{t,0}}}^2:= \sum_{A_1, \dots, A_k \in {\SM_{t,0}}} ||A_1 \dots A_k u||^2_{L^2(\R^n)}.
\label{eq:WMt0} \end{align}
Now we define 
\begin{equation}\label{eq:utilde}
\tilde u(t, z) = e^{-i|z|^2/4t} u(t, z),
\end{equation}
 as when we define the pseudoconformal transformation of $u$. Conjugating the elements of $\SM_t$ by the factor $e^{-i|z|^2/4t}$ gives the elements of $\SM_{t,0}$; consequently, we have  $u \in \modulesp^k_{\SM_t} \Longleftrightarrow  \tilde u \in \modulesp^k_{\SM_{t,0}}$ (see Proposition~\ref{prop: pulling out the oscillatory factor}). 

For $t \neq 0$, we can change to an equivalent set of generators 
\begin{align} \label{eq: generators of the M0 module}
\hat{\mathcal{M}}_{t,0}:=\{\Id, z_jD_{z_i}-z_iD_{z_j}, 2tD_{z_j}, \frac{z_j}{2t}\},
\end{align}
which is obtained by replacing the last generator in \eqref{eq: generator Mt0} by 
using 
\begin{align*}
\frac{z_j}{2t} = D_{z_j}+\frac{z_j}{2t} - \frac{1}{2t}(2tD_{z_j}),
\end{align*}
and this shows the norm equivalence between their corresponding module regularity spaces, which is uniform for large $|t| \geq c > 0$. Now we can notice that if we change spatial variable from $z$ to the `pseudoconformal' variable $\zeta = z/(2t)$, then the generators in $\mathcal{M}_{t,0}$ become 
\begin{align} \label{eq: zeta generator}
\SN := \{ \Id, {\zeta}_jD_{{\zeta}_i}-{\zeta}_iD_{
{\zeta}_j},D_{{\zeta}_j},{\zeta}_j \}. 
\end{align}
There is a corresponding module regularity space $\modulesp^k_{\SN}$ (defined as above for $\modulesp^k_{\SM_t}$, defined as above with the module $\SN$ replacing $\SM_t$)   for functions defined on $\R^n_\zeta$, which correspond exactly to the spaces $\modulesp^k_{\SM_{t, 0}}$ on the $t$-time slice of $\R^{n+1}_{t, z}$, via the $t$-dependent change of variable $\zeta = z/(2t)$. The spaces $\modulesp^k_{\SN}$ are exactly the function spaces for the final state data considered in \cite{gell2022propagation, gell2023scattering}.

Our main result is in terms of the function spaces just introduced.  Recall $\tilde u$ is defined by \eqref{eq:utilde}. 

\begin{thm}\label{thm:main1}
Suppose that $n = 1,2$ or $3$, that $p$ satisfies \eqref{eq:np}, and that $k \geq 2$. Let the potential function $V(t, z)$ be in $\ang{t}^{-1} L^\infty_t \modulesp^k_{\SM_{t,0}}$, and let $f_+ \in \modulesp^k_{\SN}$ be final state data at $t = +\infty$. Then there is a unique solution $u \in L^\infty(\R_t; \modulesp^k_{\SM_t})$ 
such that $u$ solves the PDE \eqref{eq:NLS}, with final state data $f_+$ in the sense that 
\begin{equation}\label{eq:convergencetofinalstate}
\zeta \mapsto (4\pi it)^{n/2} \tilde u(t, 2t\zeta) \text{ converges to } f_+(\zeta) \text{ as } t \to +\infty  \text{ in } \modulesp^k_{\SN}(\R^n_\zeta).
\end{equation}
In particular this implies pointwise convergence, since $\modulesp^k_{\SN}$ embeds into $L^\infty$. In the same sense, $u$ converges to a final state $f_- \in \modulesp^k_{\SN}$ as $t \to -\infty$. 
Moreover, in the weaker norm of $\modulesp^{k-2}_{\SN}$ we have a rate of convergence:
\begin{equation}\label{eq:rateofconv}\begin{gathered}
 f_\pm(\zeta) - (4\pi it)^{n/2} \tilde u(t, 2t\zeta)  = O(t^{-\gamma})  \text{ in } \modulesp^{k-2}_{\SN}, \quad t \to \pm \infty, \\
 \gamma = 1, \ n=2, 3; \quad \gamma = 1/2, \ n = 1. 
 \end{gathered}\end{equation} 
The solution $u$ is a strong solution in the sense that 
\begin{equation}\label{eq:strongsoln} 
u \in C(\R_t; \modulesp^k_{\SM_t}) \cap C^{0,1}(\R_t; \modulesp^{k-2}_{\SM_t})
\end{equation} 
and if $V$ is, in addition, continuous in time with values in $\modulesp^{k-2}_{\SM_{t,0}}$, then \eqref{eq:strongsoln} can be strengthened to 
 \begin{equation}\label{eq:strongsoln2} 
u \in C(\R_t; \modulesp^k_{\SM_t}) \cap C^{1}(\R_t; \modulesp^{k-2}_{\SM_t}).
\end{equation} 
\end{thm}

\begin{rmk} The condition on the potential may look strange, but from the point of view of linearization it is quite natural.  Proposition~\ref{prop: module regularity, multiplication rule} implies that if $u \in L^\infty_t \modulesp^k_{\SM_t}$ and $p$ is an odd integer $\geq 3$, then $|u|^{p-1} \in \ang{t}^{-n(p-2)/2} \modulesp^k_{\SM_{t,0}}$, which is contained in the function space $\ang{t}^{-1} L^\infty_t \modulesp^k_{\SM_{t,0}}$ from which the potential $V$ is selected. That is, the condition on the potential is such that it captures the potential functions that will arise from treating the equation as a linear equation with potential $|u|^{p-1}$. 
\end{rmk}

\begin{rmk} On the other hand, we point out that time-independent potentials (other than zero) are not in the function space $\modulesp^k_{\SM_{t,0}}$. Rather, potentials in this class display dispersion and decay properties analogous to scattering solutions of NLS (or powers of scattering solutions). 
\end{rmk}

\begin{rmk} 
We clarify the meaning of \eqref{eq:strongsoln}, given that the function space $\modulesp^k_{\SM_t}$ changes with time. As observed above, the spaces $\modulesp^k_{\SM_t}$ have the same underlying set, and the norms for different values of $t$ are equivalent (although not uniformly so, as $|t| \to \infty$). So, the statement \eqref{eq:strongsoln} is equivalent to 
$$
u \in C(\R_t; \modulesp^k_{\SM_0}) \cap C^{0,1}(\R_t; \modulesp^{k-2}_{\SM_0})
$$
with values in a fixed function space. 
\end{rmk}

\begin{rmk} See Theorem~\ref{thm:metric} for a generalization where we allow a metric perturbation that is compactly supported in spacetime, and such that the metric $g_t$ is nontrapping for each $t$. 
\end{rmk}

\subsection{Related literature}
There has been a vast literature on the Cauchy problem of the nonlinear Schr\"odinger equation. We make no attempt to  survey well-posedness results in low regularity, or on spaces other than $\mathbb{R}^n$ in this short introduction. Instead we focus on results in $H^1$ or higher regularity that is more closely related to our setting.
 
The Cauchy problem of the nonlinear Schr\"odinger equation in $H^1$ is studied in \cite{ginibre-velo1978NLS-D123}\cite{ginibre-velo1979NLS1}\cite{ginibre-velo1979NLS2} in this setting with `local' nonlinearity (compared with Hartree-type nonlinearity involving convolution, which is non-local). The global well-posedness of cubic NLS on $\mathbb{R}^{1+3}$ is shown in \cite{ginibre1985scattering} and one always have scattering phenomena in this case. 

When the potential is time dependent, the wave operator is constructed in \cite{howland1974stationary}\cite{davies1974timedependent} with certain decay in time condition. In the setting where the potential is majorized by a time-independent potential and has small Rollnik type norm, a unitary wave operator is shown to exist in \cite{howland1980born}, which is analogues to \cite{kato1966waveoperator} in the time-independent setting. The existence of wave operator with a time-dependent long-range (in space) potential is settled in \cite{kitada-yajima-timedependent-longrange}. The strongly continuous property of the evolution group for potentials in $L^q_tL^p_z, 0 \leq \frac{1}{q} < 1- \frac{n}{2p}$ is shown in \cite{yajima1987evolutiongroup}. 
The pointwise and Stichartz type decay estimate of evolution with time-dependent and rough potential is studied in \cite{Rodnianski-Schlag2004decay}. A large class of time-dependent potentials is treated in \cite{Soffer-Wu2203.00724, Soffer-Wu2211.00500}.

Theorem \ref{thm:main1} can be interpreted as a statement about the asymptotic completeness of the wave operator (in fact its adjoint) associated to our NLS. This is because the free evolution operator is an isometry (modulo a constant multiple) on those module regularity spaces as well. Suppose we denote the evolution operator of $P=D_t+\Delta_g+V$ and $P_0=D_t+\Delta_0$, with $\Delta_0$ being the positive Laplacian of the Euclidean metric, by $U(t),U_0(t)$ respectively, then our final state problem is aiming to show that for any $f$ in the module regularity spaces introduced in Section \ref{sec: module regularity spaces}, we can find a solution $u(t)$ to \eqref{eq:NLS} such that
\begin{equation} \label{eq: convergenec, IVP}
||u(t)- U_0(t)\mathcal{F}f||_{\modulesp^k_{\mathcal{M}_t} } \rightarrow 0,  t \to \infty,
\end{equation}
where $\mathcal{F}$ is the partial Fourier transform  we used the fact that $\mathcal{P}_0f= U_0(t)\mathcal{F}f$. Suppose we think of $u(t)$ as the solution to the initial value problem with initial value $u(0)$, and notice that $U_0(t)$ is an isometry on module regularity spaces (see \ref{eq:isometry}), then we know that \eqref{eq: convergenec, IVP} is equivalent to 
\begin{equation} \label{eq: convergenec, IVP}
||U_0(-t)U(t)u(0) - \mathcal{F}f||_{\modulesp^k_{\mathcal{M}_t} } \rightarrow 0,  t \to \infty,
\end{equation}
which is the completeness of the adjoint of the wave operator. In addition, the defocusing sign and the decay assumption on the potential prevents the existence of bound states, hence there is no projection to the absolutely continuous spectrum needed as usual in the statement of the completeness.

Our method bears similarity to Klainerman's commuting vector field approach to wave equations \cite{klainerman1985vec-wave}, which has been applied to the Schr\"odinger equation in \cite{constantin1990decay}, which proved dispersive estimates, and more recently in \cite{wong2018vec}, where dispersive estimates and Strichartz estimates are derived in a Besov type space.
The norm appearing in the analogue of the Klainerman-Sobolev inequality for the Schr\"odinger equation in \cite[Lemma~11]{wong2018vec} has some similarities to the module regularity norm in this current work, but with the difference that we take generators $z_jD_{z_i}-z_iD_{z_j}$ of  rotations into consideration.

\subsection{Links with the articles \cite{gell2022propagation} and \cite{gell2023scattering}}
This paper is intended to be a bridge between the traditional works on NLS, based on a classical PDE approach using some harmonic analysis, and a more recent microlocal approach due to Gell-Redman, Gomes and the first-named author \cite{gell2022propagation, gell2023scattering}. In these papers, the inhomogeneous problem $Pu = f$ was studied on spacetime and applied to the final state problem. One main idea in these papers is to work on function spaces over spacetime, $\R^{n+1}_{t,z}$, throughout. Pairs of Hilbert spaces, $\SX^{s, \mathrm{r_\pm}; k}(\R^{n+1})$ and $\SY^{s-1, \mathrm{r_\pm}+1; k}(\R^{n+1})$ were found (where $s$ is a Sobolev exponent, $\mathrm{r}$ is a spacetime decay order and $k$ is the module regularity order with respect to a spacetime module $\mathcal{D}$ analogous to $\mathcal{M}_t$ above) between which $P$ acts invertibly. This gives rise to two inverses $P_\pm^{-1}$ which are the advanced ($+$) and retarded ($-$) propagator. Then the final state problem was solved by a two step-process: we let $u_0$ be the solution to the final state problem for the free operator $(D_t + \Delta_0) u = 0$, which is given by $u_0 = \Poip f$, where $\Poip$ is the operator
$$
f \mapsto (2\pi)^{-n} \int e^{i (z \cdot \zeta + t |\zeta|^2)} f(\zeta) \, d\zeta.
$$
In the second step, we correct $u_0$ to the solution of $Pu = (D_t + \Delta_g + V)u = 0$ by letting $u = u_0 + P_-^{-1} (Pu)$. Choosing the correct inverse $P_-^{-1}$ (instead of $P_+^{-1}$) ensures that the final state data as $t \to +\infty$ is not changed, thus $u$ has the correct final state data as well as solving the correct PDE. 

The present work uses no microlocal analysis whatsoever, but relates to \cite{gell2023scattering} through the use of modules. This is clearest when it comes to measuring the regularity of the final state data $f$, where exactly the same module $\SN$ is used in all three articles. However, in the bulk the use of modules differs here from  \cite{gell2023scattering}, since here all our function spaces are for functions defined on a time slice $\R^n_z$, rather than spacetime $\R^{n+1}_{t,z}$. While \cite{gell2023scattering} used the full symmetry group of the free Schr\"odinger operator, which includes the generator of time-translations, $D_t$ as well as the generator of parabolic dilations, $2tD_t + z \cdot D_z$, here we only use those generators that are defined on a single time-slice, thus do not involve $t$-derivatives, that is, those in \eqref{eq: generator Mt}. In effect, what we do is take the module $\SN$ for the final state data and extend it into the interior of $\R^{n+1}$, where $t$ is finite, essentially by passing to the pseudoconformal variable $\zeta$ instead of $z$, as above. Our function spaces then are $L^p$ functions of time, with values in a module regularity space in $z$. This is much closer to the traditional approach for treating the solution $u(t, z)$. It allows us to exploit  Strichartz estimates, as well as the positive-definite energy for defocusing NLS, which seems difficult to achieve in the framework of \cite{gell2022propagation, gell2023scattering}, and this is why we are able to treat the large data problem in the present article. 

The norm on the final state data $f$, or on the pseudoconformal transform of the solution $u$ at fixed time, is related to a weighted Sobolev norm of $f$, but it is not with respect to the usual constant-coefficient vector fields, but instead with respect to `1-cusp' vector fields. In polar coordinates near infinity, these take the form $r^{-1} \partial_r$ and $r^{-1} \partial_\theta$. Compared to constant coefficient vector fields, which would be bounded multiples of $\partial_r$ and $r^{-1} \partial_\theta$, we see that the radial vector field scales differently to the angular vector fields. 
One of the results of \cite{gell2022propagation}, showing an \emph{exact} correspondence between the order of $\SN$-module regularity of $f$, and the microlocal regularity of the solution $u$ on the characteristic set and away from the radial sets, 
suggests that this weighted 1-cusp norm is particularly natural for the final state data. Further evidence for the naturalness of this norm on the final state data will appear in a forthcoming article, where the authors will show that, in the setting of Theorem~\ref{thm:metric} with compactly supported potential, the scattering map $f_- \to f_+$ is a 1-cusp Fourier integral operator. While both these just-mentioned results are heavily microlocal, the more traditional PDE-type results of the present paper provide further evidence for the belief that this norm is a particularly suitable one for studying the Schr\"odinger equation.

\subsection{Outline of the proof, and structure of the paper}
Sections 2 and 3 contain preliminary material, on the 1-cusp structure, and on module regularity spaces. We prove a multiplication result for functions in our module regularity spaces, which contains a decaying factor $\ang{t}^{-n/2}$. We can view this as the `dispersive estimate' in our approach.  We also prove a crucial Gagliardo-Nirenberg-type inequality, Proposition~\ref{prop: N[u] bound},  for products of elements of module regularity spaces, which shows that we only need to control a priori finite energy solutions $u$ in $L^{p-1}_t L^\infty_z$ (globally in time). This global $L^{p-1}_t L^\infty_z$ bound is a standard consequence of  Strichartz estimates, see for example \cite[Corollary 7.3.4]{cazenave2003semilinear}, for $p$ and $n$ satisfying \eqref{eq:np}. 

The proof of Theorem~\ref{thm:main1} starts in Section 4. In that section, we show that a solution exists on a time interval $[T, \infty)$ for sufficiently large $T$ (which will go to infinity as the norm of $f$ becomes large). This is done in via the contraction mapping theorem in the space $L^\infty_t \modulesp^k_{\SM_t}$. This argument applies to all dimensions and for $p \geq 3$ for $n \geq 2$, or $p \geq 5$ for $n=1$. It also applies to focusing as well as defocusing NLS. We prove the convergence statements in Theorem~\ref{thm:main1} for $t \to +\infty$. 

In Section 5 we use a global-in-time Strichartz estimate to show that the solution extends to all $t \in \R$, and show that there is a final state $f_-$ as $t \to -\infty$. Moreover, $f_-$ is in the same space as $f_+$, and we show the same convergence to the final state for $t \to -\infty$ as we did for $t \to +\infty$. 

The proof so far described is relatively straightforward as we exploit the exact symmetries of the propagator $e^{it\Delta_0}$, where $\Delta_0$ is the flat (positive) Laplacian. 
In Section 6, we consider the case of a metric perturbation, albeit a mild one which is compactly supported both in space and time. This assumption is more for convenience than necessity; it allows us to use results of Staffilani-Tataru \cite{staffilani2002strichartz} on the Schr\"odinger equation with variable coefficients, which relies on fundamental estimates of Doi \cite{Doi1994cauchy, Doi1996remarks}. 


\section{1-cusp structure}
We briefly describe the 1-cusp pseudodifferential algebra here. For
detailed construction and explanation, see \cite{Zachos:Thesis}\cite[Section~2]{zachos2022inverting}\cite[Section~2]{jia2024tensorial}. 

Let $M$ be a manifold with boundary with chosen boundary defining function $x$. We recall that $\mathcal{V}_{\mathrm{cu}}$ is, by definition, the space of `cusp vector fields' consisting of smooth vector fields $V$ tangent to $\partial M$ such that $Vx = O(x^2)$. This is a Lie algebra of vector fields (under commutators) and it depends on the choice of $x$ only modulo $O(x^2)$. Localized away from the boundary, it is nothing other than the set of smooth vector fields on the interior $M \setminus \partial M$. 

In local coordinates $x,y_1,...,y_{n-1}$, with $y_1,...,y_{n-1}$ parametrizing $\partial M$, they have the form
\begin{align*}
a_0(x,y)x^2D_x + \sum_{j=1}^{n-1} a_j(x,y)D_{y_j},
\end{align*}
where $a_j$ are smooth functions of their variables. Then we define the 1-cusp vector fields as cusp vector fields with one extra vanishing order at the boundary:
\begin{align*}
\mathcal{V}_{\mathrm{1c}}(M): = x\mathcal{V}_{\mathrm{sc}}(M).
\end{align*}
In local coordinates, they have the form
\begin{align*}
a_0(x,y)x^3D_x + \sum_{j=1}^{n-1} a_j(x,y)xD_{y_j}.
\end{align*}
Over $C^\infty(M)$, this generates the algebra of 1-cusp differential
operators as (locally) finite sums of finite products of these.

Polynomials of those vector fields form a differential operator algebra. A typical 1-cusp differential operator is thus of the form
\begin{align*}
\sum_{\alpha+|\beta| \leq m} a_{\alpha\beta}(x,y)(x^3D_x)^\alpha(xD_y)^\beta.
\end{align*}
Correspondingly, the frame of the 1-cusp cotangent bundle, which is denoted by $^{\mathrm{1c}}T^*X$ is given by
\begin{align*}
\frac{dx}{x^3},\frac{dy_j}{x}.
\end{align*}
Let coordinates of this bundle in this frame be $\xi_{\mathrm{1c}},\eta_{\mathrm{1c}}$; that is, $\xi_{\mathrm{1c}}$ is the symbol of $x^3 D_x$ and $\eta_{\mathrm{1c},j}$ is the symbol of $x D_{y_j}$. Then the symbol class $S^{m,l}_{\mathrm{1c} }(M)$ is defined to be the class of symbols in this coordinate system such that:
\begin{align*}
|(xD_x)^\alpha D_y^\beta D_{\xi_{\mathrm{1c}}}^\gamma D_{\eta_{\mathrm{1c}}}^\delta a(x,y,\xi_{\mathrm{1c}},\eta_{\mathrm{1c}})| \leq C_{\alpha\beta\gamma\delta} \la (\xi_{\mathrm{1c}},\eta_{\mathrm{1c}}) \ra^{m-\gamma-|\delta|}x^{-l}.
\end{align*}
Such symbols can then be quantized to 1-cusp pseudodifferential operators:
\begin{multline} \label{eq: 1-cusp quantization definition}
Au(x,y) =  \\ (2\pi)^{-n}  \int e^{i (\frac{x-x'}{x^3} \xi_{\mathrm{1c}} + \frac{y-y'}{x} \cdot \eta_{\mathrm{1c}})} a(x,y,\xi_{\mathrm{1c}},\eta_{\mathrm{1c}}) u(x',y') \frac{dx'dy'}{(x')^{n+2}}d\xi_{\mathrm{1c}}d\eta_{\mathrm{1c}}.
\end{multline}
Notice that, when $M$ is the radial compactification of $\R^n$, then the Euclidean Laplacian $\Delta$ is an element of $\Psi^{2,2}_{\oc}(M)$:
\begin{equation}\label{eq:Delta1c}
\Delta \in \Psi^{2,2}_{\oc}(\overline{\R^n}). 
\end{equation}

Next we define the ellipticity of symbols and operators.
\begin{defn}
A symbol $a \in S^{m,l}_{\mathrm{1c}}(M)$ is called elliptic if it is elliptic, in the usual sense, away from the boundary of $M$, and near the boundary, using coordinates $(x, y)$ as above and 1-cusp fibre coordinates $\xi_{\mathrm{1c}},\eta_{\mathrm{1c}}$,
\begin{multline}
\exists \, c, C > 0 \text{ such that } \la (\xi_{\mathrm{1c}},\eta_{\mathrm{1c}}) \ra \geq C \text{ or } x \leq \epsilon \\ \implies  |a(x,y,\xi_{\mathrm{1c}},\eta_{\mathrm{1c}})| \geq cx^{-l} \la (\xi_{\mathrm{1c}},\eta_{\mathrm{1c}}) \ra^m  ;
\end{multline}
its quantization $A$ is also called elliptic in this case.
\end{defn}
Under this condition, see \cite[Section~2.5]{zachos2022inverting} (which gives an even stronger semiclassical statement), its quantization $A$ has a parametrix $B \in \Psi_{\mathrm{1c}}^{-m,-l}(M)$ such that
\begin{align*}
AB-\Id, BA-\Id \in \Psi_{\mathrm{1c}}^{-\infty,-\infty}(M).
\end{align*}

We can now define the 1-cusp Sobolev spaces
$H^{s,r}_{ \mathrm{1c} }(M)$, for instance for $s \geq 0$ by choosing
$A \in \Psi_{\mathrm{1c}}^{s,0}(M)$ elliptic, and demanding
$$
u\in H^{s,r}_{\mathrm{1c}}(M) \Leftrightarrow u \in x^r L^2(M)\ \text{and}\ Au\in x^r L^2(M);
$$
here $L^2(M)$ is the $L^2$ space relative to a fixed polynomially
weighted density. In our setting, $M$ will be the radial compactification of a Euclidean space $\R^n$, so we take the Euclidean density,  which in the coordinates $(x, y)$ takes the form of a smooth nonvanishing multiple of $\frac{dx\,dy}{x^{n+1}}$. Equivalently, for $s\geq 0$
integer,
\begin{align} \label{eq: 1-cusp norm definition}
\|u\|^2_{H^{s,r}_{\mathrm{1c}}}= \|x^{-r}u\|_{L^2}^2+\sum_{j+|\alpha|\leq s}\|x^{-r}(x^3D_x)^j(xD_y)^\alpha\|_{L^2}^2
\end{align}
with the spaces for other $s$ defined via interpolation and duality. Localized away from the boundary, this space coincides with the standard Sobolev space $H^s(M)$. 

\begin{prop} For $k>\frac{n}{2},r_1,r_2 \in \R$, we have
\begin{align} \label{eq: 1-cusp multiplication}
||uv||_{ H_{\oc}^{k,r_1+r_2} } \lesssim ||u||_{ H_{\oc}^{k,r_1} }
||v||_{ H_{\oc}^{k,r_2} }.
\end{align}
\end{prop}

\begin{proof}
By considering $x^{-r_1}u,x^{-r_2}v$, we may reduce to the case where $r_1=r_2=0$. Moreover, localized away from the boundary, the result is just the standard algebra property of Sobolev spaces $H^s$ for $s > n/2$. So we may reduce to the case that both $u$ and $v$ are supported in a small collar neighbourhood of the boundary. Localizing further we may work in a single coordinate patch $(x, y)$ as above. 

We then consider the `1-cusp Fourier transform' $\Fic u$ defined by 
\begin{equation}
\Fic u (\xic, \etac) = \int e^{-i(x^{-2} \xic + x^{-1} y \cdot \etac)} u(x, y) \frac{dx  \, dy}{x^{n+2}}.
\end{equation}

An analogous proof to that for the standard Fourier transform shows that $\Fic$ is an isometry with inverse 
\begin{equation}
\Fic^{-1} G(x, y) := (2\pi)^{-n} \int \int e^{i(x^{-2} \xic + x^{-1} y \cdot \etac)} G(\xic, \etac) \, d\xic \, d\etac .
\end{equation}
Moreover, $\Fic$ is an isometry from $L^2(x^{-(n+2)}dx dy)$ to $L^2(d\xic d\etac)$. 

We compute
\begin{equation}\begin{aligned}
\Fic \Big( x D_{y_j} u \Big)(\xic, \etac) &= \eta_{\mathrm{1c}, j} \Fic u(\xic, \etac), \\
\Fic \Big( (x^3 D_x - x^2 y \cdot D_y - i x^2) u \Big) (\xic, \etac) &= -2 \xic \Fic u(\xic, \etac).
\end{aligned}\label{eq:Fic-derivs}\end{equation}
It follows from \eqref{eq:Fic-derivs} that the norm $||u||_{H^{k}_{\oc}}$, for $u$ supported in an $(x, y)$-coordinate patch in a small collar neighbourhood of the boundary,  is equivalent to a weighted norm of $\Fic u$:
\begin{align*}
||u||_{H^{k}_{\oc}} \approx ||\la (\xi_{\oc},\eta_{\oc}) \ra^k \hat{u}||_{L^2},
\end{align*}
where $\approx$ means they can bound each other up to a constant. Then the proof, via Fourier transform, of the algebra property of standard Sobolev spaces can be used verbatim to obtain the result. 
\end{proof}


\section{Module regularity spaces} \label{sec: module regularity spaces}
In this section we discuss the module regularity spaces we need for the contraction mapping argument later.

\subsection{Module regularity spaces and their properties} 
We begin with the notion of a test module of pseudodifferential operators \cite{HMV2004}. 
\begin{defn} 
A test module is a vector subspace of $\Psisc^{1,1}(\R^{n})$, the scattering pseudodifferential operators over $\R^n$, that is a module over $\Psisc^{0,0}(\R^n)$, closed under commutators, and finitely generated over $\Psi^{0,0}(\R^{n})$. 

A time-foliated test module is a family of test modules ( $\subset  \Psisc^{1,1}(\R^{n})$) parametrized by $t$, where the Lie bracket is the commutator, and for each fixed time $t$, it is equipped with a generator (depending on $t$)  set over $\Psisc^{0,0}(\R^{n})$.
\end{defn}

To alleviate the burden of notations, we will denote the module and its generator set by the same notation.
We will mainly use spaces with regularity measured by following two families of generator sets:
\begin{equation} \label{eq: generator 01}
\begin{aligned}
 \mathcal{M}_t &: \{\Id, z_jD_{z_i}-z_iD_{z_j}, 2tD_{z_j}-z_j, D_{z_j}\}\\
 \mathcal{M}_{t,0} &: \{\Id, z_jD_{z_i}-z_iD_{z_j}, 2tD_{z_j}, D_{z_j}+\frac{z_j}{2t}\}
\end{aligned} \ ,  \quad |t| \geq \frac1{2}. 
\end{equation}

More generally, let $O$ be a generator set of a (time-foliated) testing module, next we define $\modulesp^k_{O}$, the module regularity space with respect to the module generated by $O$. 
Concretely, define $O^{(k)}$ to be the set of products of generators in $O$, of degree at most $k$.
Then for non-negative integer $k$, we define $\modulesp^k_{O}$ to be the subspace of $L^2(\R^n)$ such that $Au \in L^2(\R^n)$ for all $A \in O^{(k)}$, equipped with the norm
\begin{align}
||u||_{\modulesp^k_O}^2:= \sum_{A \in O^{(k)}} ||Au||^2_{L^2(\R^n)}.
\end{align}
For the testing module case, this gives a complete normed function space; while for the time-foliated testing module case, this gives a family of complete normed function spaces, parametrized by $t$. We will use 
\begin{align}
L^p_t\modulesp^k_O(\R^{n+1}),
\end{align}
where $p = \infty$ is allowed, to denote the space of functions on $\R^{n+1}$ such that $||u(t)||_{\modulesp^k_O} \in L^p(\R)$ as a function of $t$, with norm given by the $L^p$ norm of $||u(t)||_{\modulesp^k_O}$. The case for some time intervals $I$ instead of the entire $\R$ is defined in the same manner.

The $\modulesp^k_{\mathcal{M}_{t}}$-norm and the $\modulesp^k_{\mathcal{M}_{t,0}}$-norm are equivalent, after removing the correct oscillatory factor:

\begin{prop} \label{prop: pulling out the oscillatory factor}
Suppose 
\begin{align} \label{eq: tilde u definition}
\tilde u = e^{-i\frac{|z|^2}{4t}} u,
\end{align}
then 
\begin{align}
\| u(\cdot , t) \|_{\modulesp^k_{\mathcal{M}_{t}}} = \| \tilde{u}(\cdot , t) \|_{\modulesp^k_{\mathcal{M}_{t,0}}},
\end{align}
when $|t|>1/2$.
\end{prop}

\begin{proof}
The result follows from the following intertwining identities:
\begin{align}  \label{eq: u,tilde u, intertwining}
\begin{split}
& (D_{z_j}+\frac{z_j}{2t})\Big( e^{-i\frac{|z|^2}{4t}}u \Big) =  e^{-i\frac{|z|^2}{4t}} D_{z_j}u, \\
& (z_iD_{z_j}-z_jD_{z_i})\Big(e^{-i\frac{|z|^2}{4t}} u\Big) = e^{-i\frac{|z|^2}{4t}}(z_iD_{z_j}-z_jD_{z_i})u,\\
& 2tD_{z_j}\Big(e^{-i\frac{|z|^2}{4t}}u\Big) = e^{-i\frac{|z|^2}{4t}}(2tD_{z_j}-z_j)(u).
\end{split}
\end{align}
\end{proof}

Notice that the condition that the $\Psi^{0,0}-$span of $O$ is a Lie subalgebra means commutators of generators are $\Psi^{0,0}-$coefficient combinations of generators, hence the commuting two generators in a monomial changes it by a lower order monomial. Hence we may reduce the set $O^{(k)}$ to monomials of the form $A_1^{\alpha_1}A_2^{\alpha_2}..$, where $A_1,A_2,...$ are distinct generators.

The generator set $\mathcal{M}_{t,0}$ is equivalent, for $|t| \geq 1/2$, to 
\begin{align} \label{eq: generators of the M0 module}
\hat{\mathcal{M}}_{t,0}:=\{\Id, z_jD_{z_i}-z_iD_{z_j}, 2tD_{z_j}, \frac{z_j}{2t}\}, \quad |t| \geq 1/2,
\end{align}
using 
\begin{align*}
\frac{z_j}{2t} = D_{z_j}+\frac{z_j}{2t} - \frac{1}{2t}(2tD_{z_j}),
\end{align*}
and this equivalence gives the norm equivalence between their corresponding module regularity spaces, which is uniform for $|t| \geq 1/2$. To avoid the singularity for small $t$, we define for $|t| < 1/2$ 
\begin{equation} \label{eq: generator 01 t small}
\begin{aligned}
 \mathcal{M}_t &: \{\Id, z_jD_{z_i}-z_iD_{z_j}, 2(t+1)D_{z_j}-z_j, D_{z_j}\}\\
 \mathcal{M}_{t,0} &: \{\Id, z_jD_{z_i}-z_iD_{z_j}, 2(t+1)D_{z_j}, D_{z_j}+\frac{z_j}{2(t+1)}\}
\end{aligned} \ ,  \quad |t| \leq \frac1{2}. 
\end{equation}
Here, we implicitly make use of the time-translation invariance of the operator $D_t + \Delta$. Notice that we now have two definitions of $\mathcal{M}_{\pm 1/2}$ and $\mathcal{M}_{\pm 1/2, 0}$, but this does not concern us as they lead to the same module regularity space, with equivalent norms, so we can freely move between the two different definitions. 

Now define, for $|t| \geq 1/2$, 
\begin{align*}
{\zeta}_i = \frac{z_i}{2t},
\end{align*}
then the generators in \eqref{eq: generators of the M0 module} can be expressed
\begin{align}\label{eq:zetamodule}
z_jD_{z_i}-z_iD_{z_j} = {\zeta}_j 
D_{{\zeta}_i} - {\zeta}_i D_{\zeta_j}, 2tD_{z_j}=D_{{\zeta}_j}, \frac{z_j}{2t} = {\zeta}_j.
\end{align}

We show that the module regularity norm defined using these generators is effectively an 1-cusp Sobolev norm of the radially compactified $\R^n$ in ${\zeta}-$variables:
\begin{prop}   \label{prop: module0-1cusp equivalence}
Let $\mathcal{N}$ be the generator set:
\begin{align} \label{eq: zeta generator}
\Id, {\zeta}_jD_{{\zeta}_i}-{\zeta}_iD_{
{\zeta}_j},D_{{\zeta}_j},{\zeta}_j,
\end{align}
then the norm $||g||_{\modulesp^k_{\mathcal{N}}}$
is equivalent to $||g||_{H^{k,k}_{\oc}}$, with $\R^n_{{\zeta}}$ radially compactified with boundary defining function $x=\la {\zeta} \ra^{-1}$.
\end{prop}
\begin{proof}
Under the radial compactification of $\R^n$, generators in (\ref{eq: zeta generator}) are equivalent to:
\begin{align} \label{eq: weighted 1-cusp generator}
\Id, \partial_y, x^2\partial_x, x^{-1},
\end{align}
where $y=(y_1,...,y_{n-1})$ parametrizes $\partial \bar{\R}^n$. 
Since the generators other than the identity in (\ref{eq: weighted 1-cusp generator}) multiplied by $x$ are generators defining the 1-cusp norm, the $k-$th order module regularity norm defined using (\ref{eq: weighted 1-cusp generator}) is equivalent to the $||\cdot||_{H_{\oc}^{k,k}}-$norm.
\end{proof}

Thus the $\hat{\mathcal{M}}_{t,0}-$module regularity space  is the weighted 1-cusp Sobolev space with respect to ${\zeta}$ up to a $(2t)^{n/2}$ factor in the sense that letting
\begin{align} \label{eq: definition mk u,v}
\mk{u}(t,{\zeta}) = \tilde u(t,z=2t{\zeta}),
\end{align}
then 
\begin{align} \label{eq: module-1c norm relationship}
||\tilde u(t,\cdot)||_{ \modulesp^k_{\hat{\mathcal{M}}_{t,0}}} = ||(2t)^{n/2}\mk{u}(t,\cdot)||_{H^{k,k}_{\oc}(\R^n_{{\zeta}})},
\end{align}
where the $(2t)^{n/2}$ factor is introduced by the change of volume form when we change the variables.

For $|t|$ small (for example, when $|t| \leq 1/2$), we define 
\begin{align*}
\tilde{u} = e^{-i \frac{|z|^2}{4(t+1)} }u, \quad {\zeta}_i = \frac{z_i}{2(t+1)}, 
\end{align*}
and the equivalence to the 1-cusp Sobolev spaces is then valid for $|t| \leq 1/2$. Moreover, the result of Proposition~\ref{prop: pulling out the oscillatory factor} is valid for $|t| \leq 1/2$ with this definition of $\tilde u$. 

Using the multiplication result for the 1-cusp space, we have:
\begin{prop} \label{prop: module regularity, multiplication rule}
For $k>n/2$, we have
\begin{align*}
||u(t,\cdot)v(t,\cdot)||_{ \modulesp^k_{\mathcal{M}_{t,0}} }
\lesssim \ang{t}^{-n/2} ||u(t,\cdot)||_{ \modulesp^k_{\mathcal{M}_{t,0}} } ||v(t,\cdot)||_{ \modulesp^k_{\mathcal{M}_{t,0}} }.
\end{align*}
\end{prop}

\begin{proof} For $|t| \geq 1/2$, 
this follows from (\ref{eq: module-1c norm relationship}) and (\ref{eq: 1-cusp multiplication}):
\begin{align*}
||u(t,\cdot)v(t,\cdot)||_{ \modulesp^k_{\mathcal{M}_{t,0}} }
= & (2t)^{n/2}||\mk{u}(t,\cdot)\mk{v}(t,\cdot)||_{H^{k,k}_{\oc}(\R^n_{{\zeta}})}
\\ & \leq (2t)^{n/2}||\mk{u}(t,\cdot)\mk{v}(t,\cdot)||_{H^{k,2k}_{\oc}(\R^n_{{\zeta}})}
\\ & \lesssim (2t)^{n/2}||\mk{u}(t,\cdot)||_{H^{k,k}_{\oc}(\R^n_{{\zeta}})}||\mk{v}(t,\cdot)||_{H^{k,k}_{\oc}(\R^n_{{\zeta}})}
\\ & = (2t)^{-n/2} ||u(t,\cdot)||_{ \modulesp^k_{\mathcal{M}_{t,0}} } ||v(t,\cdot)||_{ \modulesp^k_{\mathcal{M}_{t,0}} }.
\end{align*}
When $|t| \leq 1/2$, the inequalities above hold with $t$ replaced by $t+1$. 
\end{proof}

\subsection{A Gagliardo-Nirenberg type inequality for Module regularity spaces}
\label{sec: module Gagliardo-Nirenberg}
 In this section we prove the following bound for the nonlinear term:
\begin{prop}  \label{prop: N[u] bound}
For $u \in \modulesp^k_{\mathcal{M}_{t}}$, we have
\begin{align} \label{eq: N[u] bound}
|| |u|^{p-1} u ||_{\modulesp^k_{\mathcal{M}_{t}}} 
\leq C ||u||_{L^\infty(\R^n)}^{p-1}||u||_{\modulesp^k_{\mathcal{M}_{t} }},
\end{align}
where the constant $C$ is independent of $t$.
\end{prop}

Recalling  the norm equivalence obtained from (\ref{eq: u,tilde u, intertwining}), this is equivalent to
\begin{align} \label{eq: N[tilde u] bound}
|| |u|^{p-1} \tilde u ||_{\modulesp^k_{\mathcal{M}_{t,0} }} 
\leq C ||\tilde{u}||_{L^\infty(\R^n)}^{p-1}||\tilde{u}||_{\modulesp^k_{\mathcal{M}_{t,0} }},
\end{align}
when $u,\tilde{u}$ are related by (\ref{eq: tilde u definition}).
The proof of Proposition \ref{prop: N[u] bound} is postponed to the end of this section, after proving several Gaglieardo-Nirenberg type inequalities of our module regularity norm. For treatment of classical version of those results, see for example \cite[Section~13.3]{taylorPDEvol3}.

For a normed function space $\mathcal{X},k \in \Z_+$, we will use $||V^k\tb{u}||_{\mathcal{X}}$ to denote the sum of the $\mathcal{X}-$norm of all $k-$th degree monomial in (components of) 
\begin{align*}
V=(D_{{\zeta}_j},{\zeta}_iD_{{\zeta}_j}-{\zeta}_jD_{{\zeta}_i},\la {\zeta} \ra),
\end{align*} 
applied to $\tb{u}$:
\begin{align}
\| V^k\tb{u} \|_{\mathcal{X}} = \sum_{|\alpha| \leq k} \|V^\alpha \tb{u}\|_{\mathcal{X}}.
\end{align}

\begin{lmm}
\label{Lemma: GN-base case}
For $1 \leq k \leq p$, we have
\begin{align}  \label{eq: GN-base case}
||V\tb{u}||^2_{L^{2k/p}} \leq C||\tb{u}||_{L^{2k/(p-1)}}||V^2\tb{u}||_{L^{2k/(p+1)}}.
\end{align}
\end{lmm}

\begin{proof}
For $V= \la {\zeta} \ra$, it follows from H\"older's inequality. For the other two vector fields, for $f \in C_0^\infty(\R^n)$, we have
\begin{align*}
\int Vf dz = 0.
\end{align*}
This is clear for $V= 2tD_{z_j}$. For $V={\zeta}_iD_{{\zeta}_j}-{\zeta}_jD_{{\zeta}_i}$, we write the integral in ${\zeta}_i,{\zeta}_j$ in polar coordinates and this vector field becomes (the radial length times) the angular derivative, and the integral vanishes.

Notice that:
\begin{align*}
V(\tb{u}(V\tb{u})(|V\tb{u}|^{q-2})) = |V\tb{u}|^q + (q-1)\tb{u}(V^2\tb{u})|V\tb{u}|^{q-2}. 
\end{align*}
Integrating both sides, and then apply H\"older's inequality to the right hand side, we have
\begin{align*}
||V\tb{u}||_{L^q}^q \leq C||\tb{u}||_{L^{2k/(p-1)}}||V^2\tb{u}||_{L^{2k/(p+1)}}||V\tb{u}||_{L^q}^{q-2}, 
\end{align*}
where $q=\frac{2k}{p}$. Dividing both sides by 
$||V\tb{u}||_{L^q}^{q-2}$ finishes the proof.
\end{proof}
\begin{rmk}
When $p=1$, $||\tb{u}||_{L^{2k/(p-1)}}$ is interpreted as $||\tb{u}||_{L^\infty}$, and the result is unchanged; while the only change to the argument is that the step applying H\"older's inequality is made simpler: we factor $||\tb{u}||_{L^\infty}$ out first and then the remaining steps are the same.
Also, for results and proofs below in this section, whenever a zero denominator appears in the exponent of $L^{\cdot}-$space, we interpret it as the $L^\infty-$norm.
\end{rmk}

Applying (\ref{Lemma: GN-base case}) to $V^{l-1}\tb{u}$, we have
\begin{align} \label{eq: GN V l j1=1 j2=1}
||V^l\tb{u}||^2_{L^{2k/p}} \leq C||V^{l-1}\tb{u}||_{L^{2k/(p-1)}}||V^{l+1}\tb{u}||_{L^{2k/(p+1)}},
\end{align}
for $k \geq 1, 1 \leq p \leq k$. 

Next we use induction to prove:
\begin{prop}
For $j_1 \leq p \leq k+1-j_2, j_1 \leq l$, we have:
\begin{align} \label{eq: GN-l-j1,l+j2}
||V^{l}\tb{u}||_{ L^{2k/p} } \leq C ||V^{l-j_1}\tb{u}||_{L^{2k/(p-j_1)}}^{ \frac{j_2}{j_1+j_2} }
||V^{l+j_2}\tb{u}||_{L^{2k(p+j_2)} }^{ \frac{j_1}{j_1+j_2}  }.
\end{align}

\end{prop}

\begin{proof}

The case $j_1=j_2=1$ is proven in Lemma \ref{Lemma: GN-base case}.
We prove the case $(j_1,j_2) =(2,1),(1,2),(2,2)$ first.
Applying the case $j_1=j_2=1$ again, we have
\begin{align*}
||V^{l-1}\tb{u}||_{L^{2k/(p-1)}} \leq C||V^{l-2}\tb{u}||_{L^{2k/(p-2)}}^{1/2}
||V^l\tb{u}||_{L^{2k/p}}^{1/2}.
\end{align*}
Substituting in (\ref{eq: GN V l j1=1 j2=1}), we have
\begin{align*}
||V^{l}\tb{u}||_{L^{2k/p}} \leq C ||V^{l-2}\tb{u}||_{L^{2k/(p-2)}}^{1/4}
||V^l\tb{u}||_{L^{2k/p}}^{1/4} ||V^{l+1}y||_{L^{2k/(p+1)} }^{1/2},
\end{align*}
which is equivalent to 
\begin{align*}
||V^{l}\tb{u}||_{L^{2k/p}} \leq C ||V^{l-2}\tb{u}||_{L^{2k/(p-2)}}^{1/3} ||V^{l+1}y||_{L^{2k/(p+1)} }^{2/3}.
\end{align*}
The case $j_1=1,j_2=2$ is proved similarly. And the case $j_1=2,j_2=2$ is proved in the same manner as well, but substitute both 
$||V^{l-1}\tb{u}||_{L^{2k/(p-1)}}$ and $||V^{l+1}y||_{L^{2k/(p+1)} }$ using (\ref{eq: GN V l j1=1 j2=1}).

Suppose (\ref{eq: GN-l-j1,l+j2}) holds for $(j_1,j_2)$, next we prove it for $(j_1+1,j_2)$ and $(j_1,j_2+1)$, provided that they still satisfy restrictions in the proposition.

By the case $j_1=j_2=1$, we have
\begin{align*}
||V^{l-j_1}\tb{u}||_{L^{2k/(p-j_1)}} \leq C||V^{l-j_1-1}\tb{u}||^{1/2}_{L^{2k/(p-j_1-1)}}||V^{l-j_1+1}\tb{u}||^{1/2}_{L^{2k/(p-j_1+1)}}
\end{align*}
Substituting this to the induction hypothesis, we have
\begin{align} \label{eq: GN proof   03}
||V^{l}\tb{u}||_{L^{2k/p}} \leq C (||V^{l-j_1-1}\tb{u}||^{1/2}_{L^{2k/(p-j_1-1)}}||V^{l-j_1+1}\tb{u}||^{1/2}_{L^{2k/(p-j_1+1)}})^{\frac{j_2}{j_1+j_2}}
||V^{l+j_2}\tb{u}||_{ L^{2k/(p+j_2)} }^{\frac{j_1}{j_1+j_2}}.
\end{align}

Now applying the induction hypothesis again, we have
\begin{align*}
||V^{l-j_1+1}\tb{u}||_{L^{2k/(p-j_1+1)}}
\leq ||V^{l-j_1-1}\tb{u}||_{L^{2k/(p-j_1-1)}}^{ \frac{j_1-1}{j_1+1} }
||V^l\tb{u}||_{L^{2k/p}}^{\frac{2}{j_1+1}},
\end{align*}
which holds since $l-j_1+1-(l-j_1-1)=2,l-(l-j_1+1)=j_1-1<j_1$. Substituting this inequality to (\ref{eq: GN proof   03}), we have
\begin{align} \label{eq: GN proof   04}
\begin{split}
||V^{l}\tb{u}||_{L^{2k/p}} 
\leq & C ||V^{l-j_1-1}\tb{u}||^{\frac{1}{2}\times \frac{j_2}{j_1+j_2}}_{L^{2k/(p-j_1-1)}}
||V^{l-j_1-1}\tb{u}||_{L^{2k/(p-j_1-1)}}^{ \frac{1}{2} \times \frac{j_1-1}{j_1+1} \times \frac{j_2}{j_1+j_2} } 
\\&
||V^l\tb{u}||_{L^{2k/p}}^{\frac{1}{2} \times \frac{2}{j_1+1} \times \frac{j_2}{j_1+j_2}}
||V^{l+j_2}\tb{u}||_{ L^{2k/(p+j_2)} }^{\frac{j_1}{j_1+j_2}},
\end{split}
\end{align}
which is equivalent to
\begin{align*}
||V^{l}\tb{u}||_{L^{2k/p}}^{ \frac{j_1(j_1+1+j_2)}{(j_1+1)(j_1+j_2)} }
\leq C ||V^{l-j_1-1}\tb{u}||^{ \frac{j_1 j_2}{(j_1+1)(j_1+j_2)}}_{L^{2k/(p-j_1-1)}}
||V^{l+j_2}\tb{u}||_{ L^{2k/(p+j_2)} }^{\frac{j_1}{j_1+j_2}}.
\end{align*}
And this is equivalent to 
\begin{align*}
||V^{l}\tb{u}||_{L^{2k/p}}
\leq C ||V^{l-j_1-1}\tb{u}||^{ \frac{j_2}{j_1+1+j_2}}_{L^{2k/(p-j_1-1)}}
||V^{l+j_2}\tb{u}||_{ L^{2k/(p+j_2)} }^{\frac{j_1+1}{j_1+1+j_2}},
\end{align*}
which is the desired inequality with $(j_1,j_2)$ replaced by $(j_1+1,j_2)$.
The case with $(j_1,j_2)$ replaced by $(j_1,j_2+1)$ is proved in the same manner.
\end{proof}
In particular, taking $p=j_1=l$ and set $k=l+j_2$, we have the following important special case:
\begin{coro}  \label{coro- GN- Vl-Linfinity-Vk}
For $k \geq l$, we have
\begin{align*}
||V^l\tb{u}||_{L^{2k/p}} \leq C ||\tb{u}||_{L^\infty}^{1-\frac{l}{k}}
||V^k\tb{u}||_{L^2}^{l/k}.
\end{align*}
\end{coro}

Next we estimate products.
\begin{lmm}  \label{lemma: GN-product of derivatives}
For $|\beta|+|\gamma|=k$, we have
\begin{align} \label{eq: GN- product of derivatives}
||(V^\beta \tb{u}_1) (V^\gamma \tb{u}_2)||_{L^2} 
\lesssim ||\tb{u}_1||_{L^\infty}||V^k \tb{u}_2||_{L^2} + 
||V^k \tb{u}_1||_{L^2}||\tb{u}_2||_{L^\infty}.
\end{align}
\end{lmm}

\begin{proof}
Denoting $|\beta|=l,|\gamma|=m,l+m=k$, by H\"older's inequality and Corollary \ref{coro- GN- Vl-Linfinity-Vk}, we have
\begin{align*}
& ||(V^\beta \tb{u}_1) (V^\gamma \tb{u}_2)||_{L^2} 
\\ \leq & ||V^\beta \tb{u}_1||_{ L^{2k/l} } ||V^\gamma \tb{u}_2||_{ L^{2k/m} }
\\ \lesssim & ||\tb{u}_1||_{L^\infty}^{1-l/k}
||V^k \tb{u}_1||_{L^2}^{l/k} ||\tb{u}_2||_{L^\infty}^{1-m/k}||V^k\tb{u}_2||_{L^2}^{m/k}.
\end{align*}
Since $1-l/k=m/k,1-m/k=l/k$, the right hand side is 
\begin{align*}
(||\tb{u}_1||_{L^\infty}||V^k\tb{u}_2||_{L^2})^{m/k}
(||\tb{u}_2||_{L^\infty}||V^k \tb{u}_1||_{L^2})^{l/k},
\end{align*}
which is controlled by the right hand side of (\ref{eq: GN- product of derivatives}).
\end{proof}

\begin{prop} \label{prop: bound, Vk product}
For $k$ as above, we have
\begin{align*}
||V^k(\tb{u}_1\tb{u}_2)|| \lesssim ||\tb{u}_1||_{L^\infty}||V^k \tb{u}_2||_{L^2} + ||V^k\tb{u}_1||_{L^2}||\tb{u}_2||_{L^\infty}.
\end{align*}
\end{prop}

\begin{proof}
Although multiplying by $\la {\zeta} \ra$ does not satisfy the Leibniz's rule, but this does not affect the fact that $V^k(\tb{u}_1\tb{u}_2)$ is still a sum of terms of the form 
\begin{align*}
(V^\beta \tb{u}_1) (V^\gamma \tb{u}_2), |\beta|+|\gamma|=k.
\end{align*}
Applying Lemma \ref{lemma: GN-product of derivatives} finishes the proof.
\end{proof}

Now we return to the proof of Proposition \ref{prop: N[u] bound}.
The desired inequality is
\begin{align*}
||V^k ( \tilde{u}^{ \frac{p+1}{2} } \bar{\tilde{u}}^{ \frac{p-1}{2} }   ) ||_{L^2}
\lesssim ||\tilde{u}||_{L^\infty}^{p-1} ||V^k\tilde{u}||_{L^2}.
\end{align*}
This can be proved by a repeated application of  Proposition \ref{prop: bound, Vk product}. Whenever in the $||V^k(\cdot)||_{L^2}-$norm,
the power of $\tilde{u},\bar{\tilde{u}}$ is greater than one, we apply Proposition \ref{prop: bound, Vk product} to factor out one of $\tilde{u},\bar{\tilde{u}}$, and notice that 
\begin{align*}
||\tilde{u}||_{L^\infty}=||\bar{\tilde{u}}||_{L^\infty}, ||V^k\tilde{u}||_{L^2}=||V^k\bar{\tilde{u}}||_{L^2},
\end{align*}
we have the desired estimate.

\subsection{$L^p$ estimates}
Module regularity spaces have good decay properties in $L^r_z$ for $r > 2$, due to the linear growth of vector fields in the modules. We have 

\begin{lmm}\label{lem:decay}
Suppose that $W \in L^\infty_t \modulesp^1_{\SM_{t,0}}$. Then $W$ decays in suitable $L^r$ norms as $|t| \to \infty$ for suitable $r > 2$. More precisely,  for $2 < r < \infty$ when $n=1, 2$ and $2 < r \leq 6$ when $n=3$, we have 
\begin{equation}\label{eq:Lrest}
\| W(t, \cdot) \|_{L^r_z}  \leq C \ang{t}^{-n(1/2 - 1/r)}\| W \|_{L^\infty_t \modulesp^1_{\SM_{t,0}}}. 
\end{equation} 
Moreover, when $n=3$ and $|t| \geq 1$ we have 
\begin{equation}\label{eq:3/2}
W \in \ang{t}^{2/3} L^\infty_t L^{3/2}_z. 
\end{equation}
\end{lmm}

\begin{proof}
We use the pseudoconformal variable $\zeta = z/(2t)$ and express the derivatives of $W$ in terms of $\zeta$, as in \eqref{eq:zetamodule} and Proposition~\ref{prop: module0-1cusp equivalence}. We see that, written in terms of $\zeta$, we have 
\begin{equation}\label{eq:West}\begin{aligned}
\ang{\zeta} W &\in t^{-n/2} L^2(d\zeta) \\
D_{\zeta} W, R W &\in t^{-n/2} L^2(d\zeta)
\end{aligned}\end{equation}
where $R$ is any of the rotation vector fields $\zeta_j D_{\zeta_k} - \zeta_k D_{\zeta_j}$. Let $\mu$ be a measure on $\R^n_\zeta$ that is equal to $d\zeta$ for $|\zeta| \leq 1$ and $d|\zeta| d\hat \zeta$ for $|\zeta| \geq 1$. In terms of this measure, we have (discarding unnecessary powers of $\zeta$)
\begin{equation}\begin{aligned}
 W &\in t^{-n/2} L^2(d\mu) \\
D_{\zeta} (\ang{\zeta}^{(n-1)/2}W), R (\ang{\zeta}^{(n-1)/2}W) &\in t^{-n/2} L^2(d\mu),
\end{aligned}\end{equation}
where the power $\ang{\zeta}^{(n-1)/2}$ adjusts for the difference of the $L^2$ spaces with respect to $d\zeta$ and $d\mu$, and we use the fact that $n \leq 3$. 

We now apply the Sobolev embedding theorem, using the vector fields $D_{\hat \zeta}$ and $D_{|\zeta|}$ (with slight abuse of notation) and the corresponding measure $d\mu$, to obtain 
\begin{equation}
(\ang{\zeta}^{(n-1)/2}W) \in t^{-n/2} L^r(d\mu)
\end{equation}
Since $r > 2$ this implies that we can return to the Euclidean measure $d\zeta$:
\begin{equation}
W \in t^{-n/2} L^r(d\zeta).
\end{equation}
Then changing back to the measure $dz$ we pick up a factor $t^{n/r}$ and we arrive at \eqref{eq:Lrest}. 

To prove \eqref{eq:3/2} we return to the first line of \eqref{eq:West} and weaken the estimate for $n=3$ to 
\begin{equation}\label{eq:L2decay}
\ang{\zeta}^{2/3} W \in t^{-3/2} L^2(d\zeta).
\end{equation}
We then apply H\"older's inequality as follows. For $|t| \geq 1$, 
\begin{equation}\begin{gathered}
\| W(t, \cdot) \|_{L^{3/2}}^{3/2} = \int |W(t, z)|^{3/2} dz \leq \int \Big( \ang{z}^{4/3} |W(t, z)|^2  \, dz \Big)^{3/4} \Big( \int \ang{z}^{-4} \, dz \Big)^{1/4} \\
\leq C  |t|^{9/4}  \Big( t^{4/3} \int \ang{\zeta}^{4/3} |W(t, \zeta)|^2  \, d\zeta \Big)^{3/4} \\
\leq C |t|, \quad \text{ using } \eqref{eq:L2decay},
\end{gathered}\end{equation}
which shows \eqref{eq:3/2}. 
\end{proof}


\section{The contraction mapping scheme}  \label{sec: contraction mapping}
\subsection{The integral equation formulation and Strichartz-type estimate}
We prove a module regularity version of the `easy endpoint' inhomogeneous Strichartz estimate, $\| u \|_{L^\infty_t L^2z} \leq C \| F \|_{L^1_t L^2_z}$, where $u$ is the advanced or retarded solution to $(D_t + \Delta_0) u = F$. 

\begin{lmm}\label{lem:symm}
(i) Let $u(t, z)$ be the solution to the equation 
\begin{equation}\begin{aligned}
(D_t + \Delta_0) u(t, z) &= F(t, z), \quad t \geq T.  \\
u(T, z) &= g(z), \quad g \in L^2(\R^n). 
\end{aligned}\end{equation}
Assume that $F \in L^1_t L^2_z$.  Then the solution $u$ is unique, and has the integral representation (Duhamel's formula)
\begin{equation}\label{eq:Duhamel}
u(t, \cdot) = e^{-i(t-T) \Delta_0} g + i \int_T^t e^{i(t-s)\Delta_0} F(s, \cdot) \, ds. 
\end{equation}

(ii) Assume in addition that, for some $k \geq 1$, $F \in L^1_t \modulesp^k_{\SM_t}$. Then $u \in L^\infty_t \modulesp^k_{\SM_t}$, and if $A_j = A_j(t)$ is a generator of $\SM_t$ for $1 \leq j \leq k$ and the integral in \eqref{eq:Duhamel} is contained in $|t| \geq 1/2$ or $|t| \leq 1/2$, then 
\begin{multline}\label{eq:Duhamelmod}
A_1(t) \dots A_k(t) u(t, \cdot) = e^{-i(t-T) \Delta_0} A_1(T) \dots A_k(T) g \\ + i \int_T^t e^{i(t-T-s)\Delta_0} A_1(s) \dots A_k(s) F(s, \cdot) \, ds. 
\end{multline}
As a consequence, we obtain the isometry property 
\begin{equation}\label{eq:isometry}
 \| u(t) \|_{\modulesp^k_{\SM_t}} = \| g \|_{\modulesp^k_{\SM_T}} 
\end{equation}
and the  `Strichartz-type estimate' 
\begin{equation}
 \| u \|_{L^\infty_t \modulesp^k_{\SM_t}} \leq  \| g \|_{\modulesp^k_{\SM_t}} + \| F \|_{L^1_t \modulesp^k_{\SM_t}}  .
 \label{eq:Duhamelbound}\end{equation}
\end{lmm}

\begin{proof} 
(i) The formula is standard. Uniqueness follows from noting that the difference of two solutions is a global solution to $(D_t + \Delta_0) u = 0$ that vanishes for $t < T$. Mass conservation implies that $u$ vanishes identically.

(ii) For $k=0$, formula \eqref{eq:Duhamel} shows that $u \in L^\infty_t L^2_z$. For $k \geq 1$, first assume that $F$ is an $L^1$ function of $t$ with values in Schwartz functions of $z$, and $g$ is Schwartz. Then $u$ is $L^\infty$ in time with values in Schwartz functions, and we can compute 
$$
 (D_t + \Delta_0) A_1(t) \dots A_k(t) u(t, \cdot) = A_1(t) \dots A_k(t) (D_t + \Delta_0) u(t, \cdot) = A_1(t) \dots A_k(t) F(t, \cdot)
 $$
 since $D_t + \Delta_0$ commutes with the generators of $\SM_t$. By assumption, $A_1(t) \dots A_k(t) F(t, \cdot)$ is in $L^1_t L^2_z$ so applying part (i) we have $A_1(t) \dots A_k(t) u(t, \cdot) \in L^\infty_t L^2_z$. We also have an initial condition, namely $A_1(t) \dots A_k(t) u(t, \cdot)$ restricted to $t=T$ is $A_1(T) \dots A_k(T) g$ so we have \eqref{eq:Duhamelmod}. This shows that, under the Schwartz assumption as well as the restriction on the range of $t$-integration, we have \eqref{eq:Duhamelbound}. 
A standard density argument allows one to extend the result to the whole space $(g, F) \in \modulesp^k_{\SM_t} \oplus L^1_t \modulesp^k_{\SM_t}$. 

Finally to remove the restriction on the interval of $t$-integration, we apply the argument to the intervals $t \geq 1/2$, $|t| \leq 1/2$ and $t \leq -1/2$ separately, recalling that the two different sets of generators at $t = \pm 1/2$ give rise to equivalent norms. 
 \end{proof}

We next claim that the integral formulation of the final-state problem for NLS \eqref{eq:NLS} is as follows:
given final state data $f \in \modulesp^k_{\SN}$, find $u \in L^\infty_t \modulesp^k_{\SM_t}$ such that $Vu \mp |u|^{p-1} u  \in L^1_t \modulesp^k_{\SM_t}$, and 
\begin{align} \label{eq: NLS, integral version}
u(t,\cdot) = (\mathcal{P}_0f)(t,\cdot) + i \int_t^\infty  e^{-i (t-s) \Delta_0}\big(Vu(s, \cdot) \mp |u|^{p-1}u(s, \cdot)\big) \, ds
\end{align}
where the sign depends on the sign in \eqref{eq:NLS}, and we recall that the Poisson operator $\mathcal{P}_0$ is defined in \eqref{eq: Poisson definition}. This is seen as follows: first, the solution to this integral equation formally satisfies \eqref{eq:NLS}, and satisfies it classically if $f$ is sufficiently regular (for example, Schwartz). Second, the $L^1_t$ condition on $Vu \mp |u|^{p-1} u$ means that for large $t$, the $L^1_t$-norm over the  interval of integration, $[t, \infty)$, tends to zero. Therefore, the norm in $L^\infty_t \modulesp^k_{\SM_t}$ of the contribution of the second term on the RHS in \eqref{eq: NLS, integral version} tends to zero as $t \to \infty$, using \eqref{eq:Duhamelbound}. Multiplying by $e^{-i|z|^2/4t}$ this contribution tends to zero in $\modulesp^k_{\SM_{t,0}}$, and changing to the variable $\zeta = z/(2t)$ shows that the contribution to the final state data is also zero. It follows that the final state data arises only from the Poisson term, and is therefore precisely $f$. 

We define the nonlinear map 
\begin{align} \label{defn: Phi definition}
\Phi(u)(t, \cdot) = (\mathcal{P}_0f)(t,\cdot) + i \int_t^\infty  e^{-i (t-s) \Delta_0} \big(Vu(s) \mp |u|^{p-1}u(\cdot,s) \big) ds. 
\end{align}
In the next subsection, we show, for either choice of sign, that $\Phi$ is a contraction map on a certain function space on $[S,\infty) \times \R^n$ we define below, provided that $S$ is sufficiently large (depending on $f$). This will give us a solution $u$ to (\ref{eq:NLS}) on $[S,\infty) \times \R^n$.

\subsection{The contraction mapping argument}\label{subsec:contraction}
As mentioned in the introduction, we will solve the equation for the large time first.
For $S>0$ to be determined later, first we define the space on which we run the contraction mapping argument. Let
\begin{align*}
K: = ||\mathcal{P}_0f||_{L^\infty_t\modulesp^k_{\mathcal{M}_t}}([S,\infty)\times \R^n).
\end{align*}
Then we work on a ball $X$ centred at  the origin in  $L^\infty_t \modulesp^k_{\mathcal{M}_t}([S,\infty)\times \R^n)$:
\begin{align*}
X:= \{ u  \mid \; \|u\|_{L^\infty_t\modulesp^k_{\mathcal{M}_t}([S,\infty)\times \R^n)} \leq 2K \}.
\end{align*}

The space $X$ is a complete metric space with respect to the distance
\begin{align*}
{d}(u,v) =  ||u-v||_{L^\infty_t \modulesp_{\mathcal{M}_t}^k}.
\end{align*}

Using Proposition \ref{prop: pulling out the oscillatory factor}, the definition of $X$ can be rephrased as (with $\tilde{u},\tilde{v}$ as in \eqref{eq: tilde u definition}, i.e.\ with the oscillatory factor removed):
\begin{align*}
X:= \{ u  \mid   \|\tilde{u}\|_{L^\infty_t\modulesp^k_{\mathcal{M}_{t,0}}([S,\infty)\times \R^n)} \leq 2K\},
\end{align*}
with distance
\begin{align*}
{d}(u,v) =  ||\tilde{u}-\tilde{v}||_{L^\infty_t \modulesp_{\mathcal{M}_{t,0}}^k}. 
\end{align*}
Changing variable to $\zeta = z/(2t)$, that is, setting 
\begin{align*}
\tilde{ \mk{u} }(t,{\zeta}) = \tilde{u}(t,2t{\zeta}),  \quad 
\tilde{ \mk{v} }(t,{\zeta}) = \tilde{v}(t,2t{\zeta}),
\end{align*}
and using (\ref{eq: module-1c norm relationship}), we have
\begin{align*}
||u(t,z)||_{\modulesp^k_{\mathcal{M}_t}} = ||(2t)^{n/2} \tilde{\mk{u}}(t,{\zeta})||_{H_{\oc}^{k,k}(\R^n_{{\zeta}})}
=  (2t)^{n/2}|| \tilde{\mk{u}}(t,{\zeta})||_{H_{\oc}^{k,k}(\R^n_{{\zeta}})},
\end{align*}
which means we can also express the distance using the weighted 1-cusp Sobolev norm: 
\begin{align*}
{d}(u,v) =  \sup_{t \geq S} (2t)^{n/2}||\tilde{\mk{u}}-\tilde{\mk{v}}||_{H_{\oc}^{k,k}(\R^n_{{\zeta}})}.
\end{align*}
Next we prove the contraction property of $\Phi$ in (\ref{defn: Phi definition}).

\begin{prop} \label{prop: contraction property} 
For $S$ large enough, $\Phi$ maps $X$ to itself, and is a contraction:
\begin{align*}
||\Phi(u)-\Phi(v)||_{L^\infty_t \modulesp_{\mathcal{M}_t}^k} \leq \frac1{2}
||u-v||_{L^\infty_t \modulesp_{\mathcal{M}_t}^k}.
\end{align*}

\end{prop}

\begin{proof} We write $N[u]$ for $\pm |u|^{p-1} u$. 
Applying \eqref{eq:Duhamelbound} to $F=N[u]-N[v]$, we have
\begin{align*}
||\Phi(u)-\Phi(v)||_{L^\infty_t \modulesp_{\mathcal{M}_t}^k}
 \leq ||N[u]-N[v]||_{L^1_t\modulesp_{\mathcal{M}_t}^k} + ||V (u - v)||_{L^1_t\modulesp_{\mathcal{M}_t}^k},
\end{align*}
so we only need to show that
\begin{align}
||N[u]-N[v]||_{L^1_t\modulesp_{\mathcal{M}_t}^k} + ||V (u - v)||_{L^1_t\modulesp_{\mathcal{M}_t}^k} \leq \frac1{2} ||u-v||_{L^\infty_t \modulesp_{\mathcal{M}_t}^k}.
\label{eq:wts}\end{align}
It is clear that 
\begin{align*}
||N[u]-N[v]||_{L^1_t\modulesp_{\mathcal{M}_t}^k}
= ||N[\tilde{u}]-N[\tilde{v}]||_{L^1_t\modulesp_{\mathcal{M}_{t,0} }^k}
\end{align*}

Notice that
\begin{align*} 
\pm(N[\tilde{u}] - N[\tilde{v}]) = & \tilde{u}^{\frac{p+1}{2}}\bar{\tilde{u}}^{\frac{p-1}{2}} - \tilde{v}^{\frac{p+1}{2}}\bar{\tilde{v}}^{\frac{p-1}{2}}
\\ = & (\tilde{u}-\tilde{v})(\sum_{k=0}^{\frac{p-1}{2}} (\tilde{u} \bar{\tilde{u}})^{\frac{p-1}{2}-k}(\tilde{v}\bar{\tilde{v}})^k)
+ (\bar{\tilde{u}}-\bar{\tilde{v}})\tilde{u}\tilde{v}(\sum_{k=0}^{\frac{p-3}{2}} (\tilde{u} \bar{\tilde{u}})^{\frac{p-3}{2}-k}(\tilde{v}\bar{\tilde{v}})^k ).
\end{align*}

Then for fixed $t$, applying Proposition \ref{prop: module regularity, multiplication rule}, we know 
\begin{align*}
||N[\tilde{u}] - N[\tilde{v}]||_{\modulesp_{\mathcal{M}_{t,0}}^k} & \lesssim
(2t)^{-n/2} 
||\tilde{u}-\tilde{v}||_{\modulesp_{\mathcal{M}_{t,0} }^k}
(||\sum_{k=0}^{\frac{p-1}{2}} (\tilde{u} \bar{\tilde{u}})^{\frac{p-1}{2}-k}(\tilde{v}\bar{\tilde{v}})^k)||_{ \modulesp_{\mathcal{M}_{t,0}}^k } 
\\ & + ||\tilde{u}\tilde{v}(\sum_{k=0}^{\frac{p-3}{2}} (\tilde{u} \bar{\tilde{u}})^{\frac{p-3}{2}-k}(\tilde{v}\bar{\tilde{v}})^k||_{\modulesp_{\mathcal{M}_{t,0}}^k})
\\ & \lesssim (2t)^{-n(p-1)/2} 
||\tilde{u}-\tilde{v}||_{\modulesp_{\mathcal{M}_{t,0} }^k}
(||\tilde{u}||_{\modulesp_{\mathcal{M}_{t,0} }^k}+||v||_{\modulesp_{\mathcal{M}_{t,0} }^k})^{p-1}
\\ & \leq (2t)^{-n(p-1)/2} 
||\tilde{u}-\tilde{v}||_{\modulesp_{\mathcal{M}_{t,0} }^k}
(4K)^{p-1}.
\end{align*}

Thus we have
\begin{align*}
||N[u] - N[v]||_{L^1_t\modulesp_{\mathcal{M}_t}^k} & \lesssim ||u-v||_{L^\infty_t\modulesp_{\mathcal{M}_t}^k} 
\int_S^\infty t^{-(p-1)n/2}dt
\\ & \lesssim ||u-v||_{L^\infty_t\modulesp_{\mathcal{M}_t}^k} S^{1-\frac{n(p-1)}{2}}.
\end{align*}
This shows that the first term in \eqref{eq:wts} is small when $S$ is large. We next consider the second term. This can be estimated, for each fixed time $t \geq 1/2$, by  
\begin{equation}\label{eq: pot term bound}
\begin{aligned} ||V(u - v)||_{\modulesp^k_{\mathcal{M}_t}} &= \| V(\tilde u - \tilde v) \|_{\modulesp^k_{\SM_{t,0}^k}} \\
&\leq t^{-1-n/2} \| tV \|_{L^\infty_t \modulesp^k_{\SM_{t,0}^k}} \| (\tilde u - \tilde v) \|_{L^\infty_t \modulesp^k_{\SM_{t,0}^k}} \\
&\leq C t^{-1-n/2} \| (\tilde u - \tilde v) \|_{L^\infty_t \modulesp^k_{\SM_{t,0}^k}}.
\end{aligned}\end{equation}
This gives a bound on $\| V(u - v) \|_{L^1_t \modulesp^k_{\SM_t}}$ on the interval $[S, \infty)$ of $C S^{-n/2} \| (\tilde u - \tilde v) \|_{L^\infty_t \SM_{t,0}^k}$. Thus choosing $S$ sufficiently large, we achieve the bound \eqref{eq:wts}.

Next we verify that $\Phi$ maps $X$ to itself. We need to show:
\begin{align*}
\||\Phi(u)||_{L^\infty_t\tilde{\modulesp_z^k}} \leq 2K.
\end{align*}
Using similar bounds as above, we obtain 
\begin{equation}\begin{aligned}
||N[u]||_{L^1_t\modulesp_{\mathcal{M}_t}^k}  &\lesssim  
||u||_{L^\infty_t\modulesp_{\mathcal{M}_t}^k} S^{1-\frac{n(p-1)}{2}}, \\
||Vu ||_{L^1_t\modulesp_{\mathcal{M}_t}^k}  &\lesssim  
||u||_{L^\infty_t\modulesp_{\mathcal{M}_t}^k} S^{-\frac{n}{2}}.
\end{aligned}\end{equation}
Therefore, for sufficiently large $S$, we have 
\begin{equation}\label{eq:Phi estimate}\begin{aligned}
||\Phi(u)||_{L^\infty_t\modulesp_{\mathcal{M}_t}^k} 
& \leq ||\mathcal{P}_0f||_{L^\infty_t\modulesp_{\mathcal{M}_t}^k} + 
||\int_t^\infty e^{-i(t-s) \Delta} (N[u] + Vu)(\cdot,s) ds ||_{L^\infty_t\modulesp_{\mathcal{M}_t}^k}
\\ & \leq ||\mathcal{P}_0f||_{L^\infty_t\modulesp_{\mathcal{M}_t}^k} + ||N[u] + Vu||_{L^1_t\modulesp_{\mathcal{M}_t}^k}  
\\  &\leq K + C K (S^{1-\frac{n(p-1)}{2}} + S^{-\frac{n}{2}}) \\
& \leq 2K,
\end{aligned}\end{equation}
where we used the fact that under the restriction \eqref{eq:np}, the exponents of $S$ in the second last line are strictly negative. (This is one reason why we are not able to treat $n=1, p=3$.) 
\end{proof}

\subsection{Convergence}
We next address the convergence of the solution to its final state. To do so we pass to the pseudoconformal transform $U$ of $u$, defined by 
\begin{equation}
U(\tb{t}, \zeta) = (4\pi i t)^{n/2} e^{-it |\zeta|^2} u(t, 2t\zeta) = (4\pi i t)^{n/2} \mk u(\zeta, t), \quad \zeta = \frac{z}{2t}, \quad \tb{t} = \frac1{4t}. 
\end{equation}
We write $\tb{t} = 1/(4t)$ for the pseudoconformal time variable. Thus we are investigating the convergence of $U(\tb{t}, \cdot)$ to $f$ in the topology of $\modulesp^k_{\SN}$ as $\tb{t} \to 0$. 

We prove the following result, which addresses the convergence statement in Theorem~\ref{thm:main1}, and provides further details in case we have extra regularity of the final state data $f$. 

\begin{prop} \label{prop: exist solution for large time}
For $S$ large enough, there exists a unique $u \in X$ solving (\ref{eq: NLS, integral version}) on the time interval $[S, \infty)$. In particular, we have
\begin{align}\label{eq:u in module space}
||u(t)||_{\modulesp_{\mathcal{M}_t}^k} \leq 2K,
\end{align}
for any $t \in [S,\infty)$. And we have the convergence to its final state:
\begin{align} \label{eq: convergence of final state}
||(4\pi i t)^{n/2} e^{-i\frac{|z|^2}{4t}} u - f(\frac{z}{2t})||_{\modulesp_{\mathcal{M}_t}^{k} } \rightarrow 0, \; t \rightarrow \infty,
\end{align}
which is equivalent (see \eqref{eq: module-1c norm relationship}) to the statement 
\begin{align} \label{eq: convergence of final state 2}
\| U(\tb{t}, \zeta) - f(\zeta) \|_{\modulesp^k_{\SN}} = \| U(\tb{t}, \zeta) - f(\zeta) \|_{H^{k,k}_{\oc}} \to 0, \text{ as } \tb{t} \to 0. 
\end{align}

Moreover, we have a \emph{rate} of convergence of the final state as $t \to +\infty$ in the weaker topology of $\modulesp^{k-2}_{\SM_t}$:
\begin{align} \label{eq: convergence of final state, with rate}
||(4\pi i t)^{n/2} e^{-i\frac{|z|^2}{4t}} u - f(\frac{z}{2t})||_{\modulesp_{\mathcal{M}_t}^{k-2} } \leq Ct^{-\gamma}, t \in [S,\infty)
\end{align}
where $\gamma$ is as in \eqref{eq:rateofconv}, 
which in terms of weighted 1-cusp Sobolev spaces is equivalent to 
\begin{align} \label{eq: convergence of final state 2}
||(4\pi i t)^{n/2}e^{-it|\zeta|^2}u(t,2t\zeta) - f(\zeta)||_{H^{k-2,k-2}_{\oc}(\R_\zeta^n)} \leq Ct^{-\gamma}, t \in [S,\infty).
\end{align}
Moreover, in the 1-cusp formulation, we can interpolate through non-integral orders: for every $\epsilon \in [0, 2]$ we have \begin{align} \label{eq: convergence of final state, epsilon}
||(4\pi i t)^{n/2}e^{-it|\zeta|^2}u(t,2t\zeta) - f(\zeta)||_{H^{k-\epsilon,k-\epsilon}_{\oc}(\R_\zeta^n)} \leq Ct^{-\gamma\frac{\epsilon}{2}}.
\end{align}
\end{prop}

\begin{proof}
The existence and uniqueness of a solution to (\ref{eq: NLS, integral version}) satisfying \eqref{eq:u in module space}  are a direct consequence of Proposition~\ref{prop: contraction property} and the contraction mapping theorem.  

To prove the convergence statement, we consider the equality (\ref{eq: NLS, integral version}). 
Estimate \eqref{eq:Phi estimate} shows that the second term of \eqref{eq: NLS, integral version}, integrated in time from $t$ to $\infty$, is $O(t^{1-\frac{n(p-1)}{2}} + t^{-\frac{n}{2}})$ in $\modulesp^k_{\SM_t}$, which in terms of $U$ shows that the contribution of this term to $U$ is $O(t^{1-\frac{n(p-1)}{2}} + t^{-\frac{n}{2}})$ in $\modulesp^k_{\SN}$. This is a convergence rate $O(t^{-1})$ in dimensions $n=2,3$ and $O(t^{-1/2})$ in dimension 1, that is, a rate $O(t^{-\gamma})$ where $\gamma$ is as in \eqref{eq:rateofconv}. 

So now consider the Poisson term in \eqref{eq: NLS, integral version}. We change variables from $z$ to to $\zeta = z/(2t)$ and from $t$ to $\tb{t} = 1/(4t)$ and write 
\begin{multline}
(4\pi i t)^{n/2} e^{-it |\zeta|^2} \SP_0 f(t, z)  = (\pi i /\tb{t})^{n/2} e^{-i |\zeta|^2/(4\tb{t})} (2\pi)^{-n}  \int e^{-i|\zeta'|^2/(4\tb{t})} e^{i\zeta \cdot \zeta'/(2 \tb{t})} f(\zeta') \, d\zeta'  \\
= c (4\pi i \tb{t})^{-n/2} \int e^{i|\zeta - \zeta'|^2/(4\tb{t})} f(\zeta') \, d\zeta'.
\end{multline}
This merely expresses the well-known property that the pseudoconformal transformation of a solution to the free Schr\"odinger equation is itself a solutions to the free Schr\"odinger equation, up to some normalization factors \cite[Exercise 2.28]{tao2006nonlinear}. That is, 
$$
U(\tau, \cdot) = e^{i\tau \Delta} f.
$$
The strong continuity of the Schr\"odinger semigroup on $L^2$ then shows that $U(\tau, \cdot) \to f$ in $L^2$. Moreover,  the symmetries of the Schr\"odinger equation, now in $(\tau, \zeta)$-coordinates, as in Lemma~\ref{lem:symm} imply strong continuity of the Schr\"odinger semigroup on $\modulesp^k_{\SN}(\R^n_{\zeta})$. 
Thus, the convergence also happens in $\modulesp^k_{\SN}(\R^n_{\zeta})$.

Next we show (\ref{eq: convergence of final state 2}),(\ref{eq: convergence of final state, epsilon}). 
We have already seen that the second, time-integrated term of \eqref{eq: NLS, integral version} already has decay rate $O(t^{-\gamma})$ even in the stronger norm $\modulesp^k_{\SM_t}$, so we only need to treat the Poisson term. Clearly it is sufficient  to show 
\begin{align*}  
||(4\pi i t)^{n/2}e^{-i|z|^2/4t}\mathcal{P}_0f(t,2t\zeta)-f(\zeta)||_{H^{k-\epsilon,k-\epsilon}_{\oc}(\R_\zeta^n)} \leq Ct^{-\epsilon/2} ||f||_{H^{k,k}_{\oc}(\R_\zeta^n)}, \quad 0 \leq \epsilon \leq 2, 
\end{align*}
when $||f||_{H^{k,k}_{\oc}(\R_\zeta^n)}$ is finite. And then this is reduced to the case $k=2$ by the intertwining properties of the module $\modulesp^k_{\SM_t}$ with Schr\"odinger solutions. 
Again using the pseudoconformal transformation as above, this is equivalent to
\begin{align} \label{eq: estimate U, epsilon order}
||U(\tb{t})f -f||_{H^{2-\epsilon,2-\epsilon}_{\oc}(\R^n_{\zeta'})} \leq C\tb{t}^{\epsilon/2} ||f||_{H^{k,k}_{\oc}(\R_\zeta^n)}.
\end{align}
For $\epsilon = 2$, this follows from the fact that $U(\tb{t})-f=(e^{i\tb{t}\Delta}-\Id)f$, and
\begin{equation}\begin{aligned}
\frac{1}{\tb{t}}(e^{i\tb{t}\Delta}-\Id) = & \frac{1}{\tb{t}} \int_0^\tb{t} \frac{d}{ds}e^{is\Delta}ds
\\ = & i \Big( \frac{1}{\tb{t}} \int_0^t e^{is\Delta}ds \Big) \Delta.
\end{aligned}\label{eq:semigroup}\end{equation}
Since $e^{is\Delta}$ is unitary on $L^2$, and $\Delta \in \Psi^{2,2}_{\oc}(\overline{\R^n})$ is bounded  $H^{2,2}_{\oc}(\R^n) \to L^2$ (see \eqref{eq:Delta1c}), we see from  \eqref{eq:semigroup} that 
\begin{align}
||e^{i\tb{t}\Delta}-\Id||_{H^{2,2}_{\oc}(\R^n) \rightarrow L^2} \leq C \tb{t}.
\end{align}
And since 
\begin{align}
||e^{i\tb{t}\Delta}-\Id||_{H^{2,2}_{\oc}(\R^n) \rightarrow H^{2,2}_{\oc}(\R^n)} \leq C, \quad 0 \leq \tb{t} \leq S^{-1}, 
\end{align}
applying the complex interpolation method to the following family of operators $H^{2,2}_{\oc}(\R^n) \to L^2(\R^n)$
$$
\mathrm{Op}(x^{z-2} \la (\xi_{\oc},\eta_{\oc})\ra^{2-z})  \circ (e^{i\tb{t}\Delta}-\Id) 
$$ 
for each fixed $\tb{t}$, where $\mathrm{Op}(\cdot)$ is the quantization defined in \eqref{eq: 1-cusp quantization definition}, with $z$ ranging over the band in the complex plane with real part in $[0,2]$, 
we have, since $x^{z-2}\mathrm{Op}(\la (\xi_{\oc},\eta_{\oc}) \ra^{2-z}) \in \Psi^{2-z, 2-z}_{\oc}(\R^n)$ is elliptic for $z \in [0, 2]$ real, 
\begin{align}
||e^{i\tb{t}\Delta}-\Id||_{H^{2,2}_{\oc}(\R^n) \rightarrow H^{2-z,2-z}_{\oc}(\R^n)} \leq C \tb{t}^{z/2},
\end{align}
for $0 \leq z \leq 2$, which proves (\ref{eq: estimate U, epsilon order}).
\end{proof}


\section{Global properties} \label{sec: global properties}
In this section, we extend the solution we obtained in Section \ref{sec: contraction mapping} from a neighbourhood of infinity in time to the whole real line.
We assume throughout this section that we are working with the defocusing equation, that is, the minus sign in \eqref{eq:NLS}. 

\subsection{A priori estimates} 
We prove certain uniform-in-time estimates on the solution that hold for as long as the solution exists. We first record the elementary and well-known observation that:

\begin{lmm}\label{lem:masscons}
Assume that $u$ satisfies \eqref{eq:NLS}, with $V \in \ang{t}^{-1} L^\infty_t \modulesp^2_{\SM_{t, 0}}$ real-valued. Then the $L^2$ norm of $u(t, \cdot)$ is constant in time. 
\end{lmm}

The second concerns the energy, which for us is the quantity
\begin{equation}\label{eq:energy}
E(t) = \frac1{2} \int |\nabla u(t, z)|^2 \, dz + \frac1{p+1} \int |u(t, z)^{p+1} \, dz. 
\end{equation}
Notice that we do \emph{not} include a potential term $1/2 \int V |u|^2 \, dz$ in the energy. When the potential is time independent, including this term leads to a conserved quantity. However, in our case, the potential depends on time and the energy is not conserved, with or without the potential term. Another more crucial reason is that we have made no assumption on differentiability of the potential $V(t, z)$ in $t$. Thus, we would not be able to differentiate the energy with respect to time, if such a term were included.   

Although the energy $E(t)$ is not conserved, under our assumptions on $V$ it is uniformly bounded in time, as long as the solution exists:

\begin{lmm}\label{lem:energybdd}
Assume that $u$ satisfies \eqref{eq:NLS}, with $V \in \ang{t}^{-1} L^\infty_t \modulesp^2_{\SM_{t, 0}}$ real-valued. 
 Choose a positive $\SE > 0$ and  let $\tilde E(t) = \max(\SE, E(t))$. Then, $\tilde E(t)$ satisfies 
\begin{equation}
\frac{d}{dt} \tilde E(t) \leq b(t) \tilde E(t),
\end{equation}
where $b(t) = \| V(t, \cdot) \|_{\modulesp^2_{\SM_{t,0}}} \in L^1_t$. Consequently, $\tilde E(t)$, and therefore also $E(t)$, are bounded uniformly in time.
\end{lmm}

\begin{proof} Differentiating $\tilde E(t)$ in time, we obtain either $0$, if $E(t) < \SE$, or we follow the usual calculation and arrive at 
\begin{equation}\label{eq:Ederiv}
\frac{d}{dt} E(t) = \sum_{j=1}^n \Im \int D_j V(t, z) u(t, z) \overline{D_j u(t, z)} \, dz,
\end{equation}
so
\begin{equation}
\Big| \frac{d}{dt} \tilde E(t) \Big| \leq   \| \nabla V(t, \cdot) \|_{L^r} \| u(t, \cdot) \|_{L^{p+1}} \| \nabla u \|_{L^2} , \quad r = \frac{2(p+1)}{p-1}. 
\end{equation}
We can bound $\| u(t, \cdot) \|_{L^{p+1}} \leq C E(t)^{1/(p+1)}$ and $\| \nabla u \|_{L^2} \leq C E(t)^{1/2}$. Moreover, the exponent $r = 2(p+1)/(p-1)$ for the norm of $\nabla V$ is in the range $(2, 4]$. Since $V \in \ang{t}^{-1} L^\infty_t \modulesp^2_{\SM_{t, 0}}$, we have $\nabla V \in \ang{t}^{-1} L^\infty_t \modulesp^1_{\SM_{t, 0}}$. Applying Lemma~\ref{lem:decay}, we find that $\nabla V \in \ang{t}^{-1 - \epsilon} L^\infty_t L^r_z$ for some $\epsilon > 0$. Therefore, we have 
we can bound 
\begin{equation}
\Big| \frac{d}{dt} \tilde E(t) \Big| \leq C \ang{t}^{-1 - \epsilon}E^{1/2 + 1/(p+1)} \leq C \ang{t}^{-1 - \epsilon}  \tilde E(t) . 
\end{equation}
where we use the fact that $E(t) \geq \SE$ to replace $E^{1/2 + 1/(p+1)}$ with $E$ up to a constant. Thus the time derivative of $\log \tilde E(t)$ is integrable in time, so $\tilde E(t)$, and hence $E(t)$ itself, is bounded uniformly in time. 
\end{proof}

\begin{coro}\label{cor:globalH1}
There is a global-in-time solution $u(t, z) \in L^\infty_t H^1(\R^n)$ of \eqref{eq:NLS} extending the solution found in Proposition~\ref{prop: exist solution for large time}. 
\end{coro}

\begin{proof} This follows from local well-posedness of \eqref{eq:NLS} in $H^1(\R^n)$ \cite[Section 3.3]{cazenave2003semilinear}, which means that starting from an initial condition $u(T)$ in $H^1$, a solution can be found on a finite time interval $[T-c, T+c]$ where $c$ depends only on $\| u(T) \|_{H^1}$. Uniform boundedness of the energy means that this solution can be extended repeatedly with time-intervals that are uniformly bounded from below in length, and thus the solution can be extended to all $t \in \R$.
\end{proof}

The previous Corollary shows that there is a global solution, but it only provides $H^1$ regularity in space, not $\modulesp^k_{\SM_t}$ regularity as we wish to show. 
To proceed, we next assume, following the approach in \cite[Section 7]{cazenave2003semilinear}, some additional decay of the initial condition, in addition to being in $H^1$. 
The next estimate is a pointwise (in time) decay estimate on $\| u(t, \cdot) \|$, assuming that it solves \eqref{eq:NLS} with initial data $\phi \in \ang{z}^{-1} L^2(\R^n) \cap H^1(\R^n)$ given at some time $T$. Notice that this is a weaker condition than being in the module regularity space $\modulesp^1_{\SM_T}$. 

\begin{lmm}\label{lem:NLSLrdecay}
Assume that the function $u(t, z)$ solves \eqref{eq:NLS} with $n$ and $p$ satisfying \eqref{eq:np},  with real-valued potential function $V(t, z) \in \ang{t}^{-1} L^\infty_t \modulesp^2_{\SM_{t, 0}}$ and  initial data $\phi \in \ang{z}^{-1} L^2(\R^n) \cap H^1(\R^n)$ given at some time $T$. Then, for $2 \leq r \leq \infty$ in dimension $1$, $2 \leq r < \infty$ in dimension $2$ and $2 \leq r \leq 6$ in dimension 3, we have 
\begin{equation}\label{eq:NLSLrdecay}
\| u (t, \cdot) \|_{L^r_z} \leq C \ang{t}^{-n (1/2 + n/8)(1/2 - 1/r)}.
\end{equation}
\end{lmm}

\begin{rmk} The rate of decay in the third line of \eqref{eq:NLSLrdecay} is less than the decay $\ang{t}^{-n(1/2 - 1/r)}$ obtained  in \cite[Section 7]{cazenave2003semilinear} in the case with $V=0$. The reason for this is pointed out in the proof below. This estimate is not optimal, but it suffices for our purposes in the proof of Proposition~\ref{prop:Linftybound}. 
\end{rmk}

\begin{proof}
We follow \cite[Sections 7.2 and 7.3]{cazenave2003semilinear} in computing the time derivative of $h(t)$ (with some variations due to the presence of a potential in our setting), obtaining 
\begin{multline}\label{eq:hdefn}
h(t) = 4(t-T)^2 \| \nabla \tilde u(t, \cdot) \|_{L^2}^2 + \frac{8(t-T)^2}{p+1} \int |u|^{p+1} \, dz \\ = \| (z + 2i(t-T) \nabla) u(t, \cdot) \|_{L^2}^2 + \frac{8(t-T)^2}{p+1} \int |u|^{p+1} \, dz.
\end{multline}
Here, $\tilde u = e^{-i|z|^2/4(t-T)} u$, and the idea is that $\tilde u$ is smoother than $u$ at infinity as the principal oscillation has been removed, so that $z$-derivatives should give rise to additional decay $\sim 1/t$. We compute the time derivative of $h(t)$ as in \cite[Section 7.2]{cazenave2003semilinear}. (However, when comparing our identities to Cazenave's book, bear in mind that our energy is given by \eqref{eq:energy} and does not contain the potential term, nor is it constant in time, as it is in Cazenave; moreover, our potential $V$ corresponds to $-V$ in Cazenave.) We obtain 
\begin{multline}
h'(t) = -4(t-T) \frac{4-n(p-1)}{p+1} \int |u|^{p+1} - 4(t-T) \int (z \cdot \nabla V) |u|^2 \\ + 8(t-T)^2 \Im \int (\nabla V) u \cdot \nabla \overline{u}. 
\end{multline}
The last term is an extra term, compared to Cazenave, and corresponds to the nonzero time derivative of $E$ as in \eqref{eq:Ederiv}. 

Using the fundamental theorem of calculus, we have
\begin{multline} 
h(t) = h(T) + \int_T^t h'(t) \, dt = \| z \phi \|_{L^2}^2 + 4 \int_T^t 4(s-T) \Bigg\{ -\frac{(n(p-1)-4)}{p+1}  \int |u|^{p+1} \\ -   (z \cdot \nabla V) |u|^2 + 2(s-T) \Im (\nabla V) u \cdot \nabla \overline{u} \Bigg\} \, ds.
\end{multline}
By our assumptions \eqref{eq:np} on $n$ and $p$, the first term in the integrand has a favourable sign (nonpositive for $t > T$, nonnegative for $t < T$), which by itself would imply that $h(t)$ was uniformly bounded. We estimate the second and third terms in the integrand. The second term we estimate by
$$
(s-T) \| z \cdot \nabla V \|_{L^2} \| u \|_{L^{4}}^2,
$$
and note that $\| z \cdot \nabla V \|_{L^2} = O(\ang{t}^{-1})$ since $z_j D_{z_j} = (z/t)_j (t D_{z_j}) = \zeta_j D_{\zeta_j}$ are module elements of $\SM_{t,0}$, and the module norm of $V$ decays as $\ang{t}^{-1}$ by assumption. Moreover, $\| u \|_{L^4}$ can be interpolated between the conserved $L^2$ norm and the $L^{p+1}$ norm in the energy $E(t)$, which is already known to be uniformly bounded thanks to Lemma~\ref{lem:energybdd}. Thus, the second term in the integrand is uniformly bounded in $t$. The third term is similar: we estimate it by 
$$
(s-T) \| (s-T) \nabla V \|_{L^\infty} \| u \|_{L^2} \| \nabla u \|_{L^2},
$$
and, since $(s-T) \nabla$ is a module element, we know by Lemma~\ref{lem:decay} that $\| (s-T) \nabla V \|_{L^4}$ decays like $\ang{t}^{-1 - n/2}$ for large $t$, while the other factors are uniformly bounded by powers of $E(t)$.

Therefore, the integrand is uniformly bounded in time, which means that the integral, and thus $h(t)$ itself, is bounded by $C \ang{t}$. 
This can be compared to the estimate from Cazenave when $V=0$, where, as noted above, $h(t)$ would be uniformly bounded. Thus, compared to the trivial estimate $h(t) = O(\ang{t}^2)$, we have `half as much' decay, which leads to our decay rates in the second line of \eqref{eq:NLSLrdecay} being half what is proved in Cazenave in the $V=0$ case. 
From \eqref{eq:hdefn} we deduce that 
\begin{equation}\label{eq:tilde u decay 1}
\| \nabla \tilde u(t, \cdot) \|_{L^2} \leq C \ang{t-T}^{-1/2}.
\end{equation}

Then employing the interpolation inequality (fractional Gagliardo-Nirenberg) as in \cite[Section 7.3]{cazenave2003semilinear}, we have 
\begin{equation}\label{eq:fGN}
\| u(t, \cdot) \|_{L^r_z} = \| \tilde u(t, \cdot) \|_{L^r_z} \leq C \| \nabla \tilde u \|_{L^2}^{n(1/2 - 1/r)} \| \tilde u \|_{L^2}^{1 - n(1/2 - 1/r)},
\end{equation}
which with \eqref{eq:tilde u decay 1} and the uniform bound on $\| u \|_{L^2}$, yields the estimate
\begin{equation}\label{eq:NLSLrdecay v1}
\| u(t, \cdot) \|_{L^r_z} \leq  C \ang{t}^{-(n/2)(1/2 - 1/r)},
\end{equation}
which is weaker than what we claim in \eqref{eq:NLSLrdecay}. Now we run the argument again, except that we strengthen the $L^4$ bound on $u$ by using \eqref{eq:NLSLrdecay v1} with $r=4$. Using this, we  improve \eqref{eq:tilde u decay 1} to 
$$
\| \nabla \tilde u(t, \cdot) \|_{L^2} \leq C \ang{t-T}^{-1/2-n/8}.
$$
Feeding this improved decay into \eqref{eq:fGN} yields \eqref{eq:NLSLrdecay}. We remark that if $V$ vanishes, then $h(t)$ is uniformly bounded, and then the argument yields the optimal decay rate $\ang{t}^{-n(1/2 - 1/r)}$ as in Cazenave. 

\end{proof}

Our next estimate establishes that, provided that the initial condition is in $H^1(\R^n) \cap \ang{z}^{-1} L^2(\R^n)$, the solution is in $L^{q}_t L^\infty_t$. In preparation for the next Proposition we recall the condition for admissible Strichartz exponents $(q, r)$ in dimension $n$: 
\begin{equation}\label{eq:admissible}
\frac{2}{q} + \frac{n}{r} = \frac{n}{2}, \quad (q, r, n) \neq (\infty, 2, 2). 
\end{equation}

\begin{prop}\label{prop:Linftybound} Assume that the function $u(t, z)$ solves \eqref{eq:NLS} with real-valued potential function $V(t, z) \in \ang{t}^{-1} L^\infty_t \modulesp^2_{\SM_{t, 0}}$, $p$ and $n$ as in \eqref{eq:np}, and  initial data $\phi \in \ang{z} L^2(\R^n) \cap H^1(\R^n)$ given at some time $T$. Then,  we have  
\begin{equation}\label{eq:ubound}
u \in L^{p-1}_t L^\infty_z
\end{equation}
globally in time (i.e.\ for as long as the solution exists). 
\end{prop}

\begin{proof}
Without loss of generality, we work on the interval $[T, \infty)$. The same argument applies to $(-\infty, T]$. 

The argument is different in each dimension. We start with $n=3$, in which case we have $p=3$. In this case, it is easiest to deduce \eqref{eq:ubound} from the Sobolev bound 
\begin{equation}\label{eq:uSob}
u \in L^{2}_t W^{1,6}_z,
\end{equation}
using the fact that $(2, 6)$ are admissible exponents in dimension $3$. 
We argue as in \cite[Corollary 7.3.4]{cazenave2003semilinear}. The fact that $u$ is in $L^2_{t} W^{1,6}_z$ on compact time-intervals follows from standard local existence theory, see e.g. \cite[Remark~4.4.3]{cazenave2003semilinear}.  So the issue is to obtain the estimate globally in time, assuming that the solution exists on an infinite time interval. We write the solution in terms of the free propagator $e^{-it\Delta}$:
\begin{equation}
u(t) = e^{-i(t-T)\Delta} \phi(T) - i \int_T^t e^{-i(t-s)\Delta} \big(V(s)u(s) + |u|^2(s) u(s) \big) \, ds. 
\end{equation}
We commute a $z$-derivate through the free propagator. The Strichartz estimate for the free propagator shows that the term $e^{-i(t-T)\Delta} \phi(T)$ is in the desired space $L^2_t W^{1,6}_z$ globally in time. So consider the second term. We use the inhomogeneous Strichartz estimate to show that the $L^2_t W^{1,6}_z$ norm of the second term is bounded by 
$$
C \| V(s)u(s) + |u|^2(s) u(s) \|_{L^{2}_t W^{1, 6/5}_z}.
$$
We work on a finite interval $[T, T^{**}]$ that we break up into $[T, T^*] \cup [T^*, T^{**}]$. Using the triangle inequality we write 
\begin{multline}\label{eq:break}
C \| V(s)u(s) + |u|^2(s) u(s) \|_{L^{2}([T, T^{**}]), W^{1, 6/5}_z} \\ \leq C \| V(s)u(s) + |u|^2(s) u(s) \|_{L^{2}([T, T^{*}]), W^{1, 6/5}_z}  \\ + C \| V(s)u(s) + |u|^2(s) u(s) \|_{L^{2}([T^*, T^{**}]), W^{1, 6/5}_z} . 
\end{multline}
Then we have for any time interval $I$
\begin{multline}\label{eq:Vu3}
 \| V(s)u(s)  \|_{L^{2}_t(I) W^{1, 6/5}_z} \leq C \Big( \| V \|_{L^\infty_t(I) L^{3/2}_z} \| u \|_{L^2_t(I) W^{1,6}_z} \\ +  \| \nabla V \|_{L^\infty_t(I) L^{3/2}_z} \| u \|_{L^2_t(I) L^{6}_z} \Big).
\end{multline}
For $s \in [T^*, \infty)$ for large $T^*$, the norms $\| V \|_{L^\infty_t L^{3/2}_z}$ and $\| \nabla V \|_{L^\infty_t L^{3/2}_z}$ are bounded by $O(|t|^{-1/3})$, see \eqref{eq:3/2}.  The second term is estimated by 
\begin{equation}\label{eq:uuu3}
C \| |u|^2(s) u(s) \|_{L^{2}_t W^{1, 6/5}_z} \leq \| u \|_{L^\infty_t L^2_z} \| u \|_{L^\infty_t L^6_z} \| u \|_{L^{2}_t W^{1, 6}_z}.
\end{equation}
For $s \in [T^*, \infty)$ for large $T^*$, the norm $ \| u \|_{L^\infty_t L^6_z}$ is bounded by $C \ang{T^*}^{-1/2}$ as we saw in Lemma~\ref{lem:NLSLrdecay}. We now fix some sufficiently large  $T^*$, and treat the first term on the RHS of \eqref{eq:break} as a constant. This gives us  
\begin{equation}\label{eq:break2}
\| u  \|_{L^2([T, T^{**}]), W^{1, 6}_z}  \leq C  + \frac1{2} \| u  \|_{L^2([T^*, T^{**}]), W^{1, 6}_z} . 
\end{equation}
This inequality implies that $\| u  \|_{L^2([T, T^{**}]), W^{1, 6}_z}$ is bounded above as $T^{**} \to \infty$, showing that this norm is bounded globally in time. The Sobolev embedding from $W^{1,6}(\R^3)$ to $L^\infty(\R^3)$ completes the proof of \eqref{eq:ubound} in the case $n=3$. 

When $n=2$, then $p \geq 3$ is odd. First we treat the case that $p \geq 5$. In that case, we can take a Strichartz estimate with admissible exponents $(q = p-1, r)$, so $r = 2(p-1)/(p-3)$ is in the range $(2, 4]$. We apply the inhomogeneous Strichartz estimate choosing, somewhat arbitrarily, dual exponents $(4/3, 4/3)$ which are dual to admissible exponents $(4,4)$. The argument is the same as above, but the numerology is different. In place of \eqref{eq:Vu3} we have 
\begin{multline}\label{eq:Vu2}
 \| V(s)u(s)  \|_{L^{4/3}_t(I) W^{1, 4/3}_z} \leq C \Big( \| V \|_{L^s_t(I) L^{\sigma}_z} \| u \|_{L^{p-1}_t(I) W^{1,r}_z} \\ +  \| \nabla V \|_{L^s_t(I) L^{\sigma}_z} \| u \|_{L^{p-1}_t(I) L^{r}_z} \Big).
\end{multline}
Here, $s$ lies in the range $[4/3, 4]$ and $\sigma$ lies in the range $[2, 4]$. We easily check that the norm of $V$ and $\nabla V$ in the preceding displayed equation are finite, and tend to zero as the left endpoint of the interval $I$ tends to $+\infty$. Similarly, in place of \eqref{eq:uuu3} we have 
\begin{equation}\label{eq:uuu3}
C \| |u|^{p-1}(s) u(s) \|_{L^{4/3}_t(I) W^{1, 4/3}_z} \leq \| u \|^{p-1}_{L^s_t(I) L^\sigma_z} \| u \|_{L^{p-1}_t(I) W^{1, r}_z},
\end{equation}
where 
$$
\frac1{p-1} + \frac{p-1}{s} = \frac{3}{4}, \quad \frac1{r} + \frac{p-1}{\sigma} = \frac{3}{4}.
$$
It is straightforward to check that $s \geq 8$ and $\sigma \geq 8$. From Lemma~\ref{lem:NLSLrdecay} we see that $\| u(t) \|_{L^{\sigma}}$ decays faster than $\ang{t}^{-9/16}$, so this is certainly in $L^s_t$, and as the left-hand endpoint of $I$ tends to infinity, this norm tends to zero. This allows us to repeat the argument above in dimension $n=2$, provided $p \geq 5$. 

This argument does not work in the case $n=2$, $p=3$ because the endpoint, $L^2$-in-time Strichartz estimate fails in dimension $2$. However, we can show, exactly as above, that we obtain a $L^q_t W^{1,r}_z$ estimate for all admissible exponents $(q, r)$. To obtain \eqref{eq:ubound} in this case we use an elementary interpolation estimate, for example 
\begin{equation}\label{eq:interpolation88}
\| u \|_{L^\infty(\R^2)} \leq C \| u \|_{W^{1,12}(\R^2)}^{1/2} \| u \|_{L^{12}(\R^2)}^{1/2} .
\end{equation}
We then use the $L^{12/5}_t$-decay of the $W^{1,12}$ norm arising from the Strichartz argument, and the $\ang{t}^{-5/8}$ decay of the $L^12$ norm given by Lemma~\ref{lem:decay} to obtain (from H\"older's inequality) $L^2$ decay of the $L^\infty$ norm. The proof of \eqref{eq:interpolation88} is elementary: we write 
$$
\| u \|_{L^\infty}^2 \leq  \sup_x \frac1{2\pi} \int_{\R^2} \frac{\big|  \nabla (\chi(x-y) |u(x-y)|^2) \big| }{|x-y|} \, dy,
$$
where $\chi(0) = 1$ and $\chi \in C_c^1(\R^2)$ is supported, say, in the ball of radius $1$. This just arises from integrating the derivative of $\chi |u|^2(x-y)$ along each ray through $x$, and averaging. We then apply H\"older's inequality to this integral, choosing the exponent $6/5$ for $|x-y|^{-1}$,  exponent $12$ for $|u|$ and exponent $12$ for $|\nabla (\chi u)|$.  

In the case $n=1$, 
estimate \eqref{eq:ubound} follows immediately from Lemma~\ref{lem:NLSLrdecay}. 
\end{proof}

\begin{rmk}\label{rem:upperboundp} For the Strichartz argument above to work, we need to show that the solution is $L^{p-1}$ in time with values in $W^{1, r}$ in space, where $r > n$, so that we can apply Sobolev embedding and obtain $L^{p-1}_t L^\infty_z$. For this to be possible in the linear case, we require that $(p-1, r)$ are admissible exponents, where $r > n$, and this is equivalent to the inequality
$$
p < \frac{n+2}{n-2}.
$$
This explains the upper bound on $p$ in \eqref{eq:np}. It is likely that the case $n=3, p=5$, which is energy-critical, can also be done, given the global well-posedness theory is available \cite{CKSTT2008} but this is a much more delicate case and is beyond the scope of this article. For the energy-supercritical cases $n=3, p \geq 7$ it seems likely that finite-time blowup may occur, given the result of \cite{MRRS2022}; while blow-up was not proved in dimension $n=3$, there seems no reason why finite-time blowup would not also occur in this dimension. 
\end{rmk}

\begin{coro}\label{cor: global Wk bound} Suppose the global solution $u(t, \cdot)$ of Corollary~\ref{cor:globalH1} is in $\modulesp^k_{\SM_t}$, locally uniformly in $t$. Then it is in $L^\infty(\R_t; \modulesp^k_{\SM_t})$, that is, the $\modulesp^k_{\SM_t}$ norm is bounded uniformly in $t$. 
\end{coro}

\begin{proof}
The solution $u(t)$ satisfies
$$
u(t) = e^{-it\Delta} u(0) -i  \int_0^t e^{-i(t-s) \Delta} \Big( V(s) u(s) + |u(s)|^{p-1} u(s) \Big) \, ds := e^{-it\Delta} u(0) + v(t).
$$
Since $u(0) \in \modulesp^k_{\SM_0}$, we find that  $u(t) \in \modulesp^k_{\SM_t}$ with the same norm as $u(0)$. For the integral term $v(t)$, we use the Strichartz-type inequality \eqref{eq:Duhamelbound} to bound 
\begin{multline}\label{eq:vbound}
\| v \|_{L^\infty([0, T + T^*]_t;  \modulesp^2_{\SM_t})} \leq  \| V u + |u|^{p-1} u \|_{L^1([0, T + T^*]_t;  \modulesp^2_{\SM_t})} \\ \leq \| V u + |u|^{p-1} u \|_{L^1([0, T]_t;  \modulesp^2_{\SM_t})} + \| V u + |u|^{p-1} u \|_{L^1([T, T + T^*]_t;  \modulesp^2_{\SM_t})}.
\end{multline}
We take $T$ large but fixed, and then consider large variable values of $T^*$. We estimate, using Propositions~\ref{prop: module regularity, multiplication rule} and \ref{prop: N[u] bound},  
\begin{multline}\label{eq:vbound2}
\| V u + |u|^{p-1} u \|_{L^1([T, T + T^*]_t;  \modulesp^2_{\SM_t})} \leq C \Big(  \| \ang{t}^{-n/2} V \|_{L^1([T, T + T^*]_t;  \modulesp^2_{\SM_{t,0}})} \\ + \| u \|_{L^{p-1}([T, T + T^*]_t; L^\infty_z)}^{p-1} \Big)  \| u \|_{L^\infty([T, T + T^*]; \modulesp^2_{\SM_t})}. 
\end{multline}
Our assumption on $V$ implies that that $\ang{t}^{-n/2} V$ is $L^1$ in time with values in $\modulesp^2_{\SM_{t,0}}$. And Proposition~\ref{prop:Linftybound} shows that $u$ is $L^{p-1}$ in time with values in $L^\infty_z$. Therefore, for  $T$ sufficiently large, the norms $\| \ang{t}^{-n/2} V \|_{L^1([T, \infty]_t;  \modulesp^2_{\SM_t})}$ and $\| u \|_{L^{p-1}([T, \infty]_t; L^\infty_z)}^{p-1}$ are small. We therefore obtain from \eqref{eq:vbound} and \eqref{eq:vbound2}, and for $T$ large enough,
$$
\| u \|_{L^\infty([T, T + T^*]; \modulesp^2_{\SM_t})} \leq C + \frac1{2} \| u \|_{L^\infty([T, T + T^*]; \modulesp^2_{\SM_t})},
$$
which gives a uniform bound on $\| u \|_{L^\infty([T, T + T^*]; \modulesp^2_{\SM_t})}$ and hence on $\| u \|_{L^\infty([0, \infty]; \modulesp^2_{\SM_t})}$. The argument for negative time is exactly analogous. 
\end{proof}

\subsection{Global solution in $L^\infty_t \modulesp^k_{\SM_t}$} \label{sec: extend the solution}

We now have the a priori estimates, in particular \eqref{eq:ubound}, that we need in order to construct a global-in-time solution in $L^\infty_t \modulesp^k_{\SM_t}$. Notice that we have a global-in-time solution $u(t) \in L^\infty_t H^1(\R^n)$ thanks to the uniform energy bound. We will construct solutions, at first on bounded time intervals, that lie in $L^\infty_t \modulesp^k_{\SM_t}$. These coincide with the existing $L^\infty_t \modulesp^k_{\SM_t}$ solution thanks to uniqueness (see Lemma~\ref{lem:symm}). 

\begin{prop}\label{prop: extend Wk solution} Given an initial value $u(T) \in \modulesp^k_{\SM_t}$, there is a solution $u(t, z) \in L^\infty_t(I) \modulesp^k_{\SM_t}$ on $I \times \R^n$, where $I = [T_*, T]$ is a (finite or infinite) time interval. The length of $I$ is bounded below by a constant  depending only on $K := \| u(T) \|_{\modulesp^k_{\SM_t}}$, and the norm of $\ang{t}^{-1/2} V$ in $ L^1(I_t;  \modulesp^k_{\SM_{t,0}})$. 
\end{prop}

\begin{proof} We set up a functional defined on $L^\infty_t(I) \modulesp^k_{\SM_t}$ for which the solution $u$ is a fixed point. 
Any solution $u$ satisfies, by definition,
\begin{equation}\label{eq:IVP}
u(t, \cdot) = e^{-i(t - t_1)\Delta} u(t_1, \cdot) +i \int_t^{t_1} e^{-i(t-s) \Delta} \Big( V(s, \cdot) u(s, \cdot) + |u(s, \cdot)|^{p-1} u(s, \cdot) \Big) \, ds. 
\end{equation}
The solution is thus a fixed point of the map $\Phi$ defined by 
\begin{equation}\label{eq:Phi IVP}
(\Phi u)(t, \cdot) := e^{-i(t - T)\Delta} u(T, \cdot) +i \int_t^{T} e^{-i(t-T-s) \Delta} \Big( V(s, \cdot) u(s, \cdot) + |u(s, \cdot)|^{p-1} u(s, \cdot) \Big) \, ds. 
\end{equation}
We claim that $\Phi$ is a contraction on the closed ball $B$  of radius $2K$  in the space $L^\infty([T_*; T]_t; \modulesp^k_{\SM_t})$, provided that the length of the interval $I$ is small relative to $K$ and provided $I$ is chosen so that $\| \ang{t}^{-1/2}  V \|_{ L^1(I_t;  \modulesp^k_{\SM_{t,0}})}$ is sufficiently small.  To see this, we first show that $\Phi$ maps $B$ to itself. We use the Strichartz-type inequality \eqref{eq:Duhamelbound} to bound 
$$
\| \Phi(u) \|_{L^\infty_t \modulesp^k_{\SM_t}} \leq  K + C \| V u + |u|^{p-1} u \|_{L^1_t \modulesp^k_{\SM_t}}.
$$
Thus it suffices to show that $C \| V u + |u|^{p-1} u \|_{L^1_t \modulesp^k_{\SM_t}} \leq K$. Using Propositions~\ref{prop: module regularity, multiplication rule} and \ref{prop: N[u] bound}, and the definition of $B$, we estimate this by  
\begin{multline}\label{eq: L1Wk bound}
C   \| \ang{t}^{-n/2} V \|_{L^1(I_t;  \modulesp^k_{\SM_t})}  \| u \|_{L^\infty(I_t;  \modulesp^k_{\SM_t})} + \| \ang{t}^{-\frac{n}{2}(p-1)} \|_{L^1(I_t)}  \| u \|_{L^\infty(I_t;  \modulesp^k_{\SM_t})}^p  \\
\leq C \Big(  \| \ang{t}^{-n/2} V \|_{L^1(I_t;  \modulesp^k_{\SM_t})} + K^{p-1} \| \ang{t}^{-\frac{n}{2}(p-1)} \|_{L^1(I_t)} \Big)  K  
\end{multline}
and we see that, as claimed, provided that the length of $I$ is small, this is bounded by  $2K$. Thus, $\Phi$ maps $B$ into itself. 

The proof that $\Phi$ is a contraction is very similar to the proof in Proposition~\ref{prop: contraction property}; we omit the details. The key point, as with the proof just given, is that the quantity 
\begin{equation}\label{eq:key quantity}
  \| \ang{t}^{-n/2} V \|_{L^1(I_t;  \modulesp^k_{\SM_t})} + K^{p-1} \| \ang{t}^{-\frac{n}{2}(p-1)} \|_{L^1(I_t)}
\end{equation}
should be sufficiently small. 

We thus find a unique fixed point for the map $\Phi$. This is a solution to \eqref{eq:NLS} and, in particular, is a finite energy solution, hence coincides with the global solution constructed in Corollary~\ref{cor:globalH1}. 
\end{proof}

Now, putting together Corollary~\ref{cor: global Wk bound} and Proposition~\ref{prop: extend Wk solution}, we see that there is a global solution in $L^\infty_t \modulesp^k_{\SM_t}$. 

\begin{prop}\label{prop:soln-global}
There is a solution $u \in L^\infty_t \modulesp^k_{\SM_t}$ of the nonlinear Schr\"odinger equation \eqref{eq:NLS} with final state  $f \in \modulesp^k_{\SN}$ as $t \to +\infty$. 
\end{prop}

\begin{proof}
This is just a matter of putting together the Propositions proved above. First, Proposition~\ref{prop: exist solution for large time} shows that there is a solution for some time interval $[S, \infty)$. For this solution, Proposition~\ref{prop: extend Wk solution} shows that it can be extended for some time interval. Then the a priori estimate of Corollary~\ref{cor: global Wk bound} shows that the $L^\infty_t \modulesp^k_{\SM_t}$ norm stays uniformly bounded, so the time interval on which we are able to extend remains bounded below. Thus, we can repeatedly extend the solution and cover the whole time interval. \end{proof}

\begin{rmk}
In fact, we only need to iterate the construction of Proposition~\ref{prop: extend Wk solution} finitely many times, because after a finite number of steps, the norm of the key quantity \eqref{eq:key quantity} will be sufficiently small even when we take $I$ to be the remaining infinite interval. 
\end{rmk}

\subsection{Convergence to a final state as $t \to -\infty$}
\label{sec: convergence on the other end}
Let $u$ be the global solution we have obtained in $L^\infty_t\modulesp^k_{\mathcal{M}_t}(\R^{n+1})$. In this section we show that it converges to a `final' state as $t \rightarrow -\infty$.

We write the solution in terms of an initial value at $t=t_1$ as in \eqref{eq:IVP}. As we have seen in Section~\ref{sec: extend the solution}, the integral in $s$ converges all the way to $s = -\infty$. So we can write $u$ in the form 
\begin{multline}\label{eq:finalstate-}
u(t, \cdot) = e^{-i(t - t_1)\Delta} u(t_1, \cdot) +i \int_{-\infty}^{t_1} e^{-i(t-s) \Delta} \Big( V(s, \cdot) u(s, \cdot) + |u(s, \cdot)|^{p-1} u(s, \cdot) \Big) \, ds \\
-i \int_{-\infty}^{t} e^{-i(t-s) \Delta} \Big( V(s, \cdot) u(s, \cdot) + |u(s, \cdot)|^{p-1} u(s, \cdot) \Big) \, ds \\
= e^{-it\Delta} \Bigg( e^{it_1 \Delta} u(t_1, \cdot) + \int_{-\infty}^t e^{is \Delta} \Big( V(s, \cdot) u(s, \cdot) + |u(s, \cdot)|^{p-1} u(s, \cdot) \Big) \, ds \Bigg) \\ -i \int_{-\infty}^{t} e^{-i(t-s) \Delta} \Big( V(s, \cdot) u(s, \cdot) + |u(s, \cdot)|^{p-1} u(s, \cdot) \Big) \, ds .
\end{multline}

\begin{prop}\label{prop:convergence-infty} The global solution constructed in Section~\ref{sec: extend the solution} solves the final state problem as $t \to -\infty$ with final state 
\begin{equation}\label{eq:fminus}
f_- = \SF^{-1} \Bigg( e^{it_1 \Delta} u(t_1, \cdot) + \int_{-\infty}^t e^{is \Delta} \Big( V(s, \cdot) u(s, \cdot) + |u(s, \cdot)|^{p-1} u(s, \cdot) \Big) \, ds \Bigg).
\end{equation}
The final state $f_-(\zeta)$ belongs to the same space, $\modulesp^k_{\SN}(\R^n_\zeta)$, as the final state $f_+$, and the pseudoconformal transform $U(\zeta, \tb{t})$, for large negative time, of $u$ converges to $f_-$ as $\tb{t} \to 0$ in the same senses as in Proposition~\ref{prop: exist solution for large time}.
\end{prop}

\begin{proof}
We first note that $f_-$ defined by \eqref{eq:fminus} is in $\modulesp^k_{\SN}(\R^n_\zeta)$. In fact, since the inverse Fourier transform $\SF^{-1}$ interwines multiplication/differentiation/rotation operators with differentiation/multiplication/rotation operators, it is enough to show that the function in large parentheses in \eqref{eq:fminus} is in $\modulesp^k_{\SM_0}$. This follows from the Strichartz-type estimate \eqref{eq:Duhamelbound}, estimating the integral term as in the proof of Proposition~\ref{prop: extend Wk solution}. 

Now, using the fact that the Poisson operator $\SP_0$ takes the form $e^{-it \Delta} \SF$, we can write \eqref{eq:finalstate-} in the form 
\begin{equation}\label{eq:finalstate-}
u(t, \cdot) = \SP_0 f_- -  i \int_{-\infty}^t e^{-i(t-s)s \Delta} \Big( V(s, \cdot) u(s, \cdot) + |u(s, \cdot)|^{p-1} u(s, \cdot) \Big) \, ds 
\end{equation}
which shows that $u$ solves the nonlinear Schr\"odinger equation \eqref{eq:NLS} with final state $f_-$ as $t \to -\infty$. Because of this, all the results of Section~\ref{sec: contraction mapping} apply (by switching the direction of time), in particular the convergence statements of Proposition~\ref{prop: exist solution for large time}.  
\end{proof}

\subsection{Proof of Theorem \ref{thm:main1}}
Now we combine results we have proven so far to prove Theorem \ref{thm:main1}.

The existence of the solution is constructed in Propositions~\ref{prop: contraction property} and \ref{prop: exist solution for large time} for a time interval $[S, \infty)$ and is then extended to the whole real line in Propositions~\ref{prop: extend Wk solution} and \ref{prop:soln-global}. The convergence results are proved in Proposition \ref{prop: exist solution for large time} and Proposition~\ref{prop:convergence-infty}. The uniqueness of the solution follows from the method of \cite[Proposition 4.2.1]{cazenave2003semilinear}.

We next address the continuity in time of the solution constructed above. Recall from \eqref{eq: NLS, integral version} that the solution satisfies 
\begin{align*}
u(t,\cdot) = (\mathcal{P}_0f)(t,\cdot) + i \int_t^\infty  e^{-i (t-s) \Delta_0}(\pm |u|^{p-1}u(s, \cdot) + Vu(s, \cdot)))ds .
\end{align*}
We now show that $u$ is in $ C(\R_t; \modulesp^k_{\SM_t}) \cap C^{0,1}(\R_t; \modulesp^{k-2}_{\SM_t})$, and if in addition we have $V \in C(\R_t;\modulesp^{k-2}_{\SM_t})$, then we show that $u \in C^{1}(\R_t; \modulesp^{k-2}_{\SM_t})$.

First we verify that the right hand side is lies in $C(\R_t; \modulesp^{k}_{\SM_t})$. For the first term, we have following intertwining properties
\begin{align} \label{eq: Poisson intertwining property}
\begin{split}
 (z_iD_{z_j}-z_jD_{z_i})\mathcal{P}_0g & = \mathcal{P}_0((\zeta_iD_{\zeta_j}-\zeta_jD_{\zeta_i})g)=e^{it\Delta_0}\mathcal{F}((\zeta_iD_{\zeta_j}-\zeta_jD_{\zeta_i})g), \\
 (2tD_{z_j}-z_j)\mathcal{P}_0g  & = \mathcal{P}_0(D_{\zeta_j}g)
 = e^{it\Delta_0}\mathcal{F}(D_{\zeta_j}g),\\
 D_{z_j}\mathcal{P}_0g  & = \mathcal{P}_0(\zeta_jg)
 =e^{it\Delta_0}\mathcal{F}(\zeta_jg),
\end{split}
\end{align}
and notice that the generator of $\mathcal{M}_t$ on the left hand sides are equivalent to that of $\mathcal{M}_s$ for finite $t,s$, so the desired continuity follows from the continuity of $e^{it\Delta_0}$ on $L^2$. An alternative way to see this continuity is first reducing to the $L^2$-case as above and consider $\mathcal{P}_0(g)$, then apply the dominant convergence theorem (which applies to $L^2$ as well) to $(e^{is|\zeta|^2}-e^{it|\zeta|^2})e^{iz\cdot\zeta}g$, $s \rightarrow t$.

The continuity of the integral term follows from \eqref{eq:Duhamelbound} with $g=0, F = - |u|^{p-1}u(s, \cdot) + Vu(s, \cdot)$ applied to intervals $[t,t'], t' \to t$, and notice that the integrand is bounded in $\modulesp^{k}_{\SM_t}$.

Differentiating \eqref{eq: NLS, integral version} in $t$, we have
\begin{align} \label{eq: partial t u}
\partial_tu = \partial_t\mathcal{P}_0f-|u|^{p-1}u+Vu + i \int_t^\infty  e^{-i (t-s) \Delta_0}\Delta_0(- |u|^{p-1}u(s, \cdot) + Vu(s, \cdot)))ds,
\end{align}
and next we verify that the right hand side lies in $L^\infty_t\modulesp^{k-2}_{\SM_t}$, and in $C(\R_t; \modulesp^{k-2}_{\SM_t})$ if we in addition have $V \in C(\R_t;\modulesp^{k-2}_{\mathcal{M}_{t,0}})$.

For the first term, we have
\begin{align}
\partial_t \mathcal{P}_0f = (2\pi)^{-n}\int e^{-it|\zeta|^2}e^{iz \cdot \zeta}(-i|\zeta|^2f(\zeta))d\zeta = \mathcal{P}_0(-i|\zeta|^2f)
\end{align}
Since $f \in \modulesp_{\mathcal{N}}^{k}(\R_\zeta^{n})$, we know $|\zeta|^2f(\zeta) \in \modulesp_{\mathcal{N}}^{k-2}$ and $\mathcal{F}(|\zeta|^2f(\zeta)) \in \modulesp_{\mathcal{N}}^{k-2}$, and \eqref{eq: Poisson intertwining property} shows $\mathcal{P}_0(-i|\zeta|^2f) \in L^\infty_t\modulesp_{\mathcal{M}_t}^{k-2}$ since it implies $\|\mathcal{P}_0(-i|\zeta|^2f)(t,\cdot)\|_{ \modulesp_{\mathcal{M}_t}^{k-2} }$
 is equivalent to $\| |\zeta|^2f(\zeta) \|_{\modulesp_{\mathcal{N}}^{k-2}}$.
To show $\mathcal{P}_0(-i|\zeta|^2f) \in C(\R_t;\modulesp^{k-2}_{\mathcal{M}_t})$, we can apply argument above for $\mathcal{P}_0f$ with $k$ replaced by $k-2$. 

To prove the continuity of the $|u|^{p-1}u$-term, we notice that
\begin{align} 
\begin{split}
|u(t)|^{p-1}u(t)-|u(s)|^{p-1}u(s) = & u(t)^{\frac{p+1}{2}}\overline{u(t)}^{\frac{p-1}{2}} - u(s)^{\frac{p+1}{2}}\overline{u(s)}^{\frac{p-1}{2}}
\\ = & (u(t)-u(s))(\sum_{k=0}^{\frac{p-1}{2}} (u(t) \overline{u(t)})^{\frac{p-1}{2}-k}(u(s)\overline{u(s)})^k)
\\&+ (\overline{u(t)}-\overline{u(s)})u(t)u(s)(\sum_{k=0}^{\frac{p-3}{2}} (u(t) \overline{u(t)})^{\frac{p-3}{2}-k}(u(s)\overline{u(s)})^k ),
\end{split}
\end{align}
and combining Proposition \ref{prop: module regularity, multiplication rule}, Proposition \ref{prop: pulling out the oscillatory factor} and the fact $u \in L^\infty(\R_t;\modulesp_{\mathcal{M}_t}^{k}) \cap C(\R_t;\modulesp_{\mathcal{M}_t}^{k})$, 
we know  $|u|^{p-1}u \in L^\infty(\R_t;\modulesp_{\mathcal{M}_t}^{k}) \cap C(\R_t;\modulesp_{\mathcal{M}_t}^{k})$.

The fact that $Vu \in L^\infty(\R_t;\modulesp_{\mathcal{M}_t}^{k-2})$ (in fact we can conclude that it is in $L^\infty(\R_t;\modulesp_{\mathcal{M}_t}^{k})$) when $V \in \la t \ra^{-1}L^\infty_t \modulesp^k_{\mathcal{M}_{t,0}}$ follows from Proposition \ref{prop: module regularity, multiplication rule}, Proposition \ref{prop: pulling out the oscillatory factor} and the fact $u \in L^\infty(\R_t;\modulesp_{\mathcal{M}_t}^{k})$.
When we in addition have $V \in C(\R_t;\modulesp^{k-2}_{\mathcal{M}_{t,0}})$, then the fact $Vu \in C(\R_t;\modulesp^{k-2}_{\mathcal{M}_{t,0}})$ can be proven in the same manner as the proof for the $|u|^{p-1}u$-term by writing 
\begin{align}
V(t)u(t) - V(s)u(s) = V(t)(u(t)-u(s))+u(s)(V(t)-V(s)).
\end{align}
The continuity of the last integral term in \eqref{eq: partial t u} follows from \eqref{eq:Duhamelbound} with $g=0, F = \Delta_0(- |u|^{p-1}u(s, \cdot) + Vu(s, \cdot))$ applied to intervals $[t,t'], t' \to t$.

We finally show that the time derivative $D_t u$ is in  $L^\infty_t \modulesp^{k-2}_{\mathcal{M}_t}$. For the Poisson term this has already been shown. For the integral term, it follows from applying \eqref{eq:Duhamelbound} as above and then noticing that 
\begin{align*}
||\Delta_0(- |u|^{p-1}u(s, \cdot) + Vu(s, \cdot))||_{\modulesp^{k-2}_{\mathcal{M}_t}} \leq ||- |u|^{p-1}u(s, \cdot) + Vu(s, \cdot)||_{\modulesp^{k}_{\mathcal{M}_t}}.
\end{align*}
Then an argument similar to \eqref{eq: L1Wk bound} shows the desired $L^\infty_t$-property.

\section{Metric perturbation}\label{sec:metric}

In this section we discuss the modification needed to derive our main results in the presence of a metric perturbation. For simplicity, and to conform with assumptions from previous articles \cite{staffilani2002strichartz} and \cite{gell2022propagation, gell2023scattering}, we assume that the 
metric perturbation is compactly supported in spacetime. Concretely, we assume that the metric has the form
\begin{align} \label{eq: metric with perturbation}
g(t) = g_{\mathrm{Euc}} + \tilde{g}(t),
\end{align}
where $g_{\mathrm{Euc}}$ is the Euclidean metric on $\R^n$ for each fixed $t$, and $\tilde{g}(t)$ is compactly supported in $t$ with values in smooth, compactly supported 2-cotensor fields such that $g(t)$ is strictly positive definite and nontrapping for each $t$. Since $\supp \tilde{g}$ is compact, we know there exist $c,C>0$ such that
\begin{align}
cg_0 \leq g \leq Cg_0,
\end{align}
where the inequalities are understood to mean that $Cg_0-g$ and $g-cg_0$ are positive definite. Hence the volume form, and corresponding $L^p-$norms of $g(t)$ and $g_{\mathrm{Euc}}$ are equivalent. 

In this section we will prove 

\begin{thm}\label{thm:metric} 
The conclusions of Theorem~\ref{thm:main1} continue to hold with the flat Laplacian $\Delta$ replaced by the metric Laplacian $\Delta_{g(t)}$ for a family of metrics $g(t)$ as described above. 
\end{thm}

\subsection{Results of Doi and Staffilani-Tataru}\label{subsec:DST} We will use the following results from the literature. 
First, it is well known that there is a solution operator $S(t, t') : L^2(\R^n) \to L^2(\R^n)$ such that $u(t) := S(t, t') g$ solves the equation 
\begin{equation}
(D_t + \Delta_{g(t)}) u(t) = 0, \quad u(t') = g \in L^2(\R^n). 
\end{equation}

Moreover, if $F \in L^1_t L^2_z$, then the solution to the inhomogeneous equation 
\begin{equation}
(D_t + \Delta_{g(t)}) u(t) = F(t), \quad u(T) = g
\end{equation}
is unique, and has the integral representation (Duhamel's formula)
\begin{equation}\label{eq:Duhamel-metric}
u(t, \cdot) = S(t, T) g + i \int_T^t S(t, s) F(s, \cdot) \, ds. 
\end{equation}

Results of Doi \cite{Doi1994cauchy, Doi1996remarks} show that the equation  
$$
(D_t + \Delta_g) u = 0, \quad u_{t = T_0} = u_0 \in H^s(\R^n)
$$
satisfies on any time interval $[T_0, T_1]$ and for any $\chi \in C_c^\infty(\R^n)$
$$
u \in C([T_0, T_1]; H^s(\R^n), \quad \chi u \in L^2([T_0, T_1]; H^{s+1/2} (\R^n).
$$
Second, Staffilani-Tataru \cite{staffilani2002strichartz} showed that if $v$ has spatial support on a fixed compact set and satisfies the equation 
$$
(D_t + \Delta_g) v = g \in L^2([T_0, T_1]; H^{-1/2}(\R^n), \quad v |_{t = T_0} = v_0 \in L^2(\R^n),
$$
then 
\begin{equation}\label{eq:ST}
v \in L^2([T_0, T_1]; H^{1/2}(\R^n) \cap L^q([T_0, T_1]; L^r(\R^n)),
\end{equation}
where $(q, r)$ are admissible, i.e.\ satisfy the conditions \eqref{eq:admissible}. 

Putting these results together, we show that if the initial condition $u_0$ for $u$ is in $H^1(\R^n)$, then 
$$
u \in L^q([T_0, T_1]; W^{1,r}(\R^n)).
$$
In fact, we consider the equation for $v_1 = W u$, where $W$ is any smooth vector field with compact support. Then $v_1$ satisfies
$$
(D_t + \Delta_g) v_1 = [D_t + \Delta_g, W] u = Q u , \quad v_1 |_{t = T_0} = W u_0 \in L^2(\R^n),
$$
where $Q = [\Delta_g, W]$ is a second-order differential operator with compactly supported coefficients. By the results of Doi, we have $Qu \in L^2_t H^{-1/2}(\R^n)$. Applying \eqref{eq:ST} we have 
$$
v_1 \in  L^q([T_0, T_1]; L^r(\R^n))
$$
for admissible exponents $(q, r)$. 

On the other hand, we can consider $v_2 = \chi D_j u$ where $\chi$ is supported where the metric is flat. Then $v_2$ solves the inhomogeneous equation with respect to the flat Laplacian, 
$$
(D_t + \Delta) v_2 = [D_t + \Delta, \chi D_j] u = Q_2 u ,\quad v_2 |_{t = T_0} = \chi D_j u_0 \in L^2(\R^n)
$$
where $Q_2$ is another second-order differential operator with compactly supported coefficients. We have, as before, $Q_2u \in L^2_t H^{-1/2}(\R^n)$. Then, we can write 
$$
v_2(t) = e^{i(t-T_0) \Delta} \chi D_j u_0 + \int_{T_0}^t e^{i(t-T_s) \Delta} Q_2 u(s) \, ds,
$$
and using the Christ-Kiselev lemma we find that, for admissible exponents $(q, r)$ such that $q > 2$, 
$$
v_2 \in L^q([T_0, T_1]; L^r(\R^n)).
$$
Overall, then, we see that for admissible exponents $(q, r)$ such that $q > 2$,
\begin{equation}
(D_t + \Delta) u = 0, \quad u |_{t = t_0} = u_0 \in H^1(\R^n) \Leftrightarrow u \in L^q([T_0, T_1]; W^{1,r}(\R^n)).
\end{equation}

\subsection{Extending the linear solution through the metric perturbation}
The beginning of the proof of Theorem~\ref{thm:metric} is exactly as for Theorem~\ref{thm:main1}, since there is no metric perturbation when $|t|$ is large. Therefore, the construction of Section~\ref{subsec:contraction} is unaffected. 

Our discussion in Section~\ref{subsec:DST} shows that the solution thus constructed extends to a global solution in $L^\infty_t H^1(\R^n)$. Our next task is to show that the solution is also in $L^\infty_t \modulesp^k_{\SM_{t,0}}$. This is certainly true on the time interval $[T_1, \infty)$ where the metric perturbation is supported in time in the interval $[T_0, T_1]$. So we consider extending the equation to the time interval $[T_0, T_1]$ with initial condition to match at $T_1$ with the solution for large $t$. We do this first for the equation $(D_t + \Delta_g) u = 0$, and then in a separate step, solve the full nonlinear Schr\"odinger equation  \eqref{eq:NLS}. 

\begin{prop}  \label{prop: module Strichartz, with perturbation}
Suppose $w$ solves the linear Schr\"odinger equation
\begin{align} \label{eq: linear Schrodinger}
\begin{split}
& (D_t+\Delta_g)w = 0,\\
& w|_{t=t_0} = w_0,
\end{split}
\end{align}
then we have
\begin{align} \label{eq: module strichartz, with metric perturbation}
||w(t,\cdot)||_{\modulesp^k_{\mathcal{M}_t}} \leq C(t,t_0)||w(t_0,\cdot)||_{\modulesp^k_{\mathcal{M}_{t_0}}}.
\end{align}
\end{prop}

\begin{rmk}
The constant in (\ref{eq: module strichartz, with metric perturbation}) may grow exponentially in time. However, this is unimportant as we only need (\ref{eq: module strichartz, with metric perturbation}) on the finite time interval $[T_0, T_1]$. Outside this time interval, the argument in Section \ref{sec: extend the solution} applies.
\end{rmk}

\begin{proof} As we have shown in Lemma~\ref{lem:symm}, we have $||w(t,\cdot)||_{\modulesp^k_{\mathcal{M}_t}}=||w(t_0,\cdot)||_{\modulesp^k_{\mathcal{M}_{t_0}}}$ when there are no metric or potential perturbations.

Now we use an energy estimate to prove (\ref{eq: module strichartz, with metric perturbation}). As before, we choose $\chi \in C_c^\infty(\R^{n+1})$ such that $\chi=1$ on $\supp \tilde{g} \cup \supp V$. Then the norm we need to control can be decomposed as
\begin{align} \label{eq:decompose u}
||w(t,\cdot)||_{\modulesp^k_{\mathcal{M}_t}} \leq
||\chi w(t,\cdot)||_{\modulesp^k_{\mathcal{M}_t}} +||(1-\chi)w(t,\cdot)||_{\modulesp^k_{\mathcal{M}_t}} 
\end{align}
Notice that on $\supp \chi$, $|z_i|,t$ are bounded, 
the module regularity norm defined by generators $\modulesp^k_{\mathcal{M}_t}$ is equivalent to the $H^k-$norm. We have
\begin{align*}
||w(t,\cdot)||_{\modulesp^k_{\mathcal{M}_t}} \lesssim
||\chi w(t,\cdot)||_{H^k_{\mathcal{M}_t}} +||(1-\chi)w(t,\cdot)||_{\modulesp^k_{\mathcal{M}_t}}.
\end{align*}
Notice that
\begin{align*}
||\chi w(t,\cdot)||_{H^k} \lesssim ||w(t,\cdot)||_{H^k},
\end{align*} 
and we can use \cite[Theorem~1.1]{Doi1994cauchy} or \cite[Remarks~1 and 3]{staffilani2002strichartz} to bound it by
\begin{align} \label{eq: Hk norm estimate, with perturbation}
||w(t,\cdot)||_{H^k} \leq C(t,t_0) ||w(t_0,\cdot)||_{H^k}.
\end{align}

For the term $||(1-\chi)w(t,\cdot)||_{\modulesp^k_{\mathcal{M}_t}}$, we consider
\begin{align*}
E_{V_1,...,V_k}(t):= \int |V_1...V_k(1-\chi)w|^2 dz,
\end{align*}
where $V_1,...,V_k$ are generators of $\mathcal{M}_t$ as in (\ref{eq: generator 01}).

Its time derivative is
\begin{align*}
\frac{d}{dt}E_{V_1,..V_k} = 2 \mathrm{Re} \int \partial_t \big(V_1...V_k(1-\chi)w \big) \, \overline{V_1...V_k(1-\chi)w} \,   dz.
\end{align*}

Since $w$ solves (\ref{eq: linear Schrodinger}), and since $\Delta_g = \Delta$ on the support of $1 - \chi$, we have
\begin{align*}
\partial_t(V_1...V_k(1-\chi)w)  = & V_1...V_k(1-\chi)\partial_tw + [\partial_t,V_1...V_k(1-\chi)] w
\\ = & -iV_1...V_k(1-\chi) \Delta w+[\partial_t,V_1...V_k(1-\chi)] w
\\ =  -i \Big(\Delta V_1...V_k(1-\chi) + &[V_1..V_k(1-\chi),\Delta]  \Big)w +[\partial_t,V_1...V_k(1-\chi)] w.
\end{align*}

Since $\Delta$ is self-adjoint, we have 
\begin{align*}
\mathrm{Re} (-i)\int \Delta V_1...V_k(1-\chi))w \, \overline{V_1...V_k(1-\chi)w} \,  dz =0.
\end{align*}
The remaining terms combine to give 
\begin{align*}
\frac{d}{dt}E_{V_1,..V_k} = & \Re \int 
\Big( i([ \Delta, V_1..V_k(1-\chi)]-[\partial_t,V_1...V_k(1-\chi)] w) \Big) \, \overline{V_1...V_k(1-\chi)w} \,  dz
\\ & = - \Im  \int 
\Big( [D_t+\Delta,V_1...V_k(1-\chi)] w \Big) \, \overline{V_1...V_k(1-\chi)w} \,  dz.
\end{align*}
We recall that $[D_t + \Delta, V_i] = 0$ for all $V_i$ in the module $\SM_t$. Consequently, we obtain 
\begin{align*}
\frac{d}{dt}E_{V_1,..V_k} =
\mathrm{Re} \int 
 i(V_1...V_k[D_t+\Delta,(1-\chi)] w) \, \overline{V_1...V_k(1-\chi)w} \,  dz.
\end{align*}
We apply Cauchy-Schwarz to this inequality, and we obtain, since there are $k+1$ derivatives on $w$ with compactly supported coefficients, and $k$ derivatives on the right hand side, 
\begin{equation}
\Big| \frac{d}{dt}E_{V_1,..V_k} \Big| \leq C  \| \tilde \chi w \|_{H^{k+1/2}}^2
\end{equation}
where $\tilde \chi \in C_c^\infty$ is identically $1$ on the support of $\nabla \chi$. Thus we have, for all $t \in [T_0, T_1]$, 
\begin{equation}
E_{V_1,..V_k}(t) \leq E_{V_1,..V_k}(T_1) + C \int_t^{T_1}  \| \tilde \chi w \|_{H^{k+1/2}}^2.
\end{equation}
Using the Doi local smoothing estimate, this is bounded by $E_{V_1,..V_k}(T_1) + C \int_t^{T_1}  \| w(T_1) \|_{H^{k}}^2$. Summing over all possible choices of $V_1, \dots, V_k$ we obtain \eqref{eq: module strichartz, with metric perturbation}. 
\end{proof}

Using this we can state the following version of Lemma~\ref{lem:symm} that holds for our metric perturbation:

\begin{lmm}\label{lem:symm-metric} Consider a metric perturbation as described above. 
Assume that $g \in \modulesp^k_{\SM_T}$ and  $F \in L^1_t \modulesp^k_{\SM_t}$ for some $k \geq 1$. Then $u \in L^\infty_t \modulesp^k_{\SM_t}$, and we have on any finite interval $[T, T^*] \ni t$
\begin{equation}\label{eq:isometry}
 \| S(t, T) g \|_{\modulesp^k_{\SM_t}} \leq C \| g \|_{\modulesp^k_{\SM_t}} 
\end{equation}
and the  `Strichartz-type estimate' 
\begin{equation}
 \| u \|_{L^\infty([T, T^*];  \modulesp^k_{\SM_t})} \leq  C \Big( \| g \|_{\modulesp^k_{\SM_t}} + \| F \|_{L^1([T, T^*];  \modulesp^k_{\SM_t})} \Big) .
 \label{eq:Duhamelbound-metric}\end{equation}
\end{lmm}

The proof of this is immediate from Proposition~\ref{prop: module Strichartz, with perturbation} and Duhamel's formula \eqref{eq:Duhamel-metric} for the solution $u$.

\subsection{Solving the full nonlinear equation \eqref{eq:NLS} through the metric perturbation}
For the full nonlinear equation 
\begin{equation}\label{eq:NLS-metric}
(D_t + \Delta_{g(t)} + V) u = - |u|^{p-1} u,
\end{equation}
we have global mass and energy bounds, where the mass and energy are now
\begin{equation}\begin{aligned}
M_g(t) &= \int |u(z, t)|^2 \rho(z,t) \, dz, \\
E_g(t) &= \frac1{2} \int \sum_{j,k} g^{jk}(z, t) \partial_j u(z, t) \partial_k u(z, t) \, \rho(z, t) \, dz + \frac1{p+1} \int |u(z,t)|^{p+1} \, \rho(z, t) \, dz. 
\end{aligned} \end{equation}
Here, $\rho dz$ is the Riemannian measure. It is straightforward to check that on the interval $[T_0, T_1]$ supporting the metric perturbation, we have  
\begin{equation}\begin{aligned}
\Big| \frac{d}{dt} M_g(t) \Big|  &\leq C M_g(t),   \\
\Big| \frac{d}{dt} E_g(t) \Big|  &\leq  C E_g(t).
\end{aligned} \end{equation}
It follows that these have at most exponential growth or decay with rate $e^{\pm Ct}$, and are therefore uniformly comparable over the time interval $[T_0, T_1]$. Off this interval, the mass is constant and the argument in Lemma~\ref{lem:energybdd} shows that the energy is uniformly bounded. 

Using these bounds we can apply Kato's method, see e.g.\ \cite[Section 4.4]{cazenave2003semilinear}, to obtain a bound on the Strichartz norm  
$$
\| u \|_{L^q([T_0, T_1]; W^{1, r}_z}
$$
for any finite energy solution to \eqref{eq:NLS-metric}. This only requires the corresponding estimate for the propagator $S(t, t')$ for the linear equation $(D_t + \Delta_{g(t)})u = 0$, as we showed in Section~\ref{subsec:DST}, and the global mass and energy bounds. Then, using the Sobolev embedding from $W^{1,r}(\R^n)$ to $L^\infty(\R^n)$, for $r > n$, we see that we have $\| u(t) \|_{L^\infty}$ in $L^q([T_0, T_1])$ for any admissible exponents $(q, r)$ with $r > n$. In particular this allows us to conclude that $\| u(t) \|_{L^\infty}$ in $L^{p-1}([T_0, T_1])$, for any pair $(n, p)$ satisfying \eqref{eq:np}. 

We next observe that, for any initial condition in $\modulesp^k_{\mathcal{M}_t}$, the solution is in $L^\infty \modulesp^k_{\mathcal{M}_t}$. We do this by showing that we can construct the solution as a fixed point of a map on 
$L^\infty(I_t; \modulesp^k_{\mathcal{M}_t})$ where we have a uniform lower bound on the time interval $I$. We omit the details, as the argument is essentially as in Section~\ref{sec: extend the solution}. The key point is that we have the $L^{p-1}_t L^\infty_z$ estimate, plus the Strichartz-type estimate \eqref{eq:Duhamelbound-metric}, required for the argument to proceed. 

\begin{proof}[Proof of Theorem~\ref{thm:metric}] We begin as for the proof of Theorem~\ref{thm:main1}, by solving the final state problem for large time. We then extend the solution, using the procedure outlined in this subsection, until we have covered the time interval on which the metric perturbation is supported. Then we solve for the remaining time interval as in Propositions~\ref{prop:soln-global} and \ref{prop:convergence-infty} and obtain convergence to the final state as $t \to -\infty$.
\end{proof}


\section{Conclusion}

In conclusion we mention several research directions which are suggested by the current article. 

$\bullet$ It would be desirable to have a framework to treat a wider class of potentials including time-independent potentials with suitable smoothness and decay. This is equally true for the spacetime Fredholm approach in the works \cite{gell2022propagation} and \cite{gell2023scattering}, where the 3-scattering calculus of Vasy \cite{Vasy20003body}, or the 3b-calculus of Hintz \cite{Hintz3b}, likely offer a valuable technique for doing this. 

$\bullet$ The assumption that the metric perturbation is compactly supported in spacetime is unnatural, and should preferably be replaced by a decay condition. Moreover, it would be valuable to replace the smoothness assumption on the metric with a finite-regularity assumption, which would open the door to studying a suitable class of quasilinear Schr\"odinger equations in module regularity spaces. 

$\bullet$  One could also consider Schr\"odinger equations in an asymptotically conic geometry instead of on $\R^n$, as in \cite{HZ2016}. 

$\bullet$  It would be desirable to go beyond the case of polynomial nonlinearities, and to treat the case of nonlinearities $f(|u|^2) u$ for suitable smooth functions $f$, with $u$ in a module regularity space. Similarly, it would be interesting to treat derivative nonlinearities, such as $|\nabla u|^2 u$. A small data theory for such nonlinearities was given in \cite{gell2023scattering}. 

$\bullet$ It would be desirable to make more serious use of the 1-cusp structure than has been done in the present article. This should allow one to make sense of a fractional order of module regularity, and thereby sharpen some of the assumptions on module regularity. 


\bibliographystyle{plain}
\bibliography{SchLargeData-arXiv-version}

\end{document}